\documentclass[12pt]{article}
\input{my_header.sty}
\usepackage{geometry}
\theoremstyle{plain}
\newtheorem{theorem}{Theorem}[section]
\newtheorem{lemma}[theorem]{Lemma}
\newtheorem{claim}[theorem]{Claim}
\newtheorem{corollary}[theorem]{Corollary}
\newtheorem{proposition}[theorem]{Proposition}

\newtheorem{example}[theorem]{Example}

\theoremstyle{definition}
\newtheorem{definition}[theorem]{Definition}
\newtheorem{remark}[theorem]{Remark}

\newtheorem{observation}[theorem]{Observation}

\numberwithin{equation}{section}

\usepackage{authblk}
\renewcommand*{\Affilfont}{\normalsize\normalfont}

\title{Planted clique detection and recovery from the hypergraph adjacency matrix\footnote{Supported by the Vilho, Yrj\"{o} and Kalle V\"{a}is\"{a}l\"{a} Foundation of the Finnish Academy of Science and Letters.}}
\author[1]{Kalle Alaluusua}
\author[2]{B. R. Vinay Kumar}
\affil[1]{ Aalto University, Espoo, Finland}
\affil[ ]{\Affilfont \href{mailto:kalle.alaluusua@aalto.fi}{kalle.alaluusua}@aalto.fi}
\affil[3]{Indian Institute of Technology Bombay (IITB), Mumbai, India}
\affil[ ]{\Affilfont \href{mailto:vinaykumar.br@iitb.ac.in}{vinaykumar.br@iitb.ac.in}}
\date{14 April 2026}
\renewcommand\Affilfont{\itshape\small}

\begin{document}
\maketitle
\begin{abstract}
Hypergraph data are often projected onto a weighted graph by constructing an adjacency matrix whose $(i,j)$ entry counts the number of hyperedges containing both nodes $i$ and $j$. This reduction is computationally convenient, but it can lose information: distinct hypergraphs may induce the same matrix, and the matrix entries are generally dependent because each hyperedge contributes to multiple pairs. We study the planted clique problem under this matrix-only observation model. For detection, we show that a spectral norm test is asymptotically powerful at the $\sqrt{n}$ scale, with explicit dependence on the background hyperedge probability $p$. For recovery, we analyze a polynomial-time spectral method based on the leading eigenvector and prove exact recovery at the canonical $\sqrt{n}$ scale, again with explicit dependence on $p$. We also extend both results to sparse regimes in which the hyperedge probability may depend on \(n\). Our analysis adapts a leave--one--out eigenvector framework to this setting. These results provide rigorous detection and recovery guarantees when only the adjacency matrix is observed.
\end{abstract}
\tableofcontents

\section{Introduction}
Hypergraphs provide a natural representation of higher-order relations among individual elements. Such higher-order structure appears in many networked datasets, including protein--protein interaction networks, brain networks, and citation networks~\cite{Battiston2021Physics,bick2023higher}. Formally, a hypergraph $H=(V,E)$ consists of a vertex set $V$ and a hyperedge set $E$, where each hyperedge is a subset of $V$.  A \emph{$d$-uniform hypergraph} on the vertex set $V=[n] := \{1,\dots,n\}$ can be represented by the collection
\[
H=(H_e)_{e\in \binom{[n]}{d}}, 
\qquad 
H_e \in \{0,1\},
\]
indexed by unordered sets of $d$ elements from $V$, where $H_e=1$ if and only if $e\in E$. Equivalently, one may view $(H_e)$ as a symmetric $d$-way $\{0,1\}$ array on $[n]^d$ by setting $H_{i_1,\ldots,i_d}=H_{\{i_1,\ldots,i_d\}}$ when the indices are distinct, and $H_{i_1,\ldots,i_d}=0$ otherwise. We refer to this representation as the \emph{adjacency tensor} of $H$.

Working directly with a hypergraph is often computationally expensive and memory-intensive. Common data representations of a $d$-uniform hypergraph on $n$ vertices include an $n\times \binom{n}{d}$ incidence matrix or a list of hyperedges, which can be wasteful in the sparse and dense regimes, respectively. A common way to obtain a simpler lower-order summary is to replace the hypergraph by a weighted graph via a clique expansion, or more generally a graph projection~\cite{boccaletti2023structure}. We refer to the resulting $n\times n$ matrix as the \emph{adjacency matrix} of the hypergraph $H$. It is also commonly called the similarity matrix or node co-occurrence matrix, since its $(i,j)$th entry records the number of hyperedges in which nodes $i$ and $j$ co-occur. In particular, we define the adjacency matrix $A$ of $H$ by
\begin{equation} 
A_{ij}\coloneqq \sum_{e:\,\{i,j\}\subset e} H_e, \quad \text{for $i\neq j$, and}\quad
A_{ii}=0.
\label{eq:def_adjacency_matrix}
\end{equation}
This reduction yields a pairwise matrix summary that supports spectral methods operating directly on $A$. Such methods can remain consistent even when the underlying hyperedges are strongly dependent, as in geometry-driven higher-order constructions~\cite{alaluusua2025consistent}.

That said, passing from the hypergraph to its adjacency matrix may discard information about the underlying higher-order structure. In particular, several hypergraphs can have the same adjacency matrix \cite{larock2025exploring} as in Example~\ref{ex:same_adj_mat} below.
\begin{example}
     On the vertex set $V=\{1,2,\cdots,8\}$, consider two $4$-uniform hypergraphs with edge sets $$E_1=\{
    \{1,2,3,5\},
    \{1,2,4,6\},
    \{3,6,7,8\},
    \{4,5,7,8\}\}
    $$
    and
    $$
    E_2=\{
    \{1,2,3,6\},
    \{1,2,4,5\},
    \{3,5,7,8\},
    \{4,6,7,8\}\},$$ respectively. Both $H_1=(V,E_1)$ and $H_2=(V,E_2)$ have the same adjacency matrix 
    $$
A =
\begin{bmatrix}
0 & 2 & 1 & 1 & 1 & 1 & 0 & 0 \\
2 & 0 & 1 & 1 & 1 & 1 & 0 & 0 \\
1 & 1 & 0 & 0 & 1 & 1 & 1 & 1 \\
1 & 1 & 0 & 0 & 1 & 1 & 1 & 1 \\
1 & 1 & 1 & 1 & 0 & 0 & 1 & 1 \\
1 & 1 & 1 & 1 & 0 & 0 & 1 & 1 \\
0 & 0 & 1 & 1 & 1 & 1 & 0 & 2 \\
0 & 0 & 1 & 1 & 1 & 1 & 2 & 0
\end{bmatrix}.
$$
\label{ex:same_adj_mat}
\end{example}
This makes it non-trivial to infer properties of the hypergraph from the adjacency matrix alone. Nevertheless, 
\cite{bresler2024thresholds,bresler2025partial} study recovery of a hypergraph from its graph projection in the $d$-uniform setting; in particular, \cite{bresler2025partial} establishes exact and partial recovery when hyperedges are observed with probability $p=O(n^{-(d-1-\delta)})$, where $\delta>0$ depends only on $d$. In a more general setting, \cite{morgan2026achievability} studies a heterogeneous random hypergraph model with hyperedges of multiple degrees and proves an achievability result for recovering degree-$d_j$ hyperedges from the projected graph by selecting maximal cliques of size $d_j$, with success measured by requiring an expected symmetric-difference error to be negligible relative to the number of true degree-$d_j$ hyperedges. Their density condition extends the $d$-uniform threshold, and they conjecture that it is optimal for recovering hyperedges of the largest degree. More broadly, recent work develops quantitative criteria for when higher-order models can be reduced to lower-order representations while preserving task-relevant information, highlighting the trade-off between descriptive power and model complexity \cite{dumitriu2026optimal,dumitriu2025partial,Lucas2026Reducibility,valimaa2025consistent}.

In this work, we study how much statistical information is lost when a hypergraph is observed only through its adjacency matrix. As a canonical example, we consider the planted clique problem in a hypergraph under this projected observation model. We begin by defining the statistical model and the inference tasks of interest.

\begin{definition}[Hypergraph planted clique model]\label{def:HPC}
Fix $d\ge 3$, $0 \le k \le n$, and $p\in[0,1]$. 
Sample $S\subset[n]$ uniformly among all subsets of $[n]$ of size $k$, and independent random variables \(\{H_e\}_{e\in\binom{[n]}{d}}\) such that
\[
\P(H_e = 1) \weq 
\begin{cases}
1, & e\subset S,\\
p, & e\not\subset S.
\end{cases}
\]
Construct the adjacency matrix \(A \in \R^{n\times n}\) as
\[
A_{ij}\coloneqq \sum_{e:\,\{i,j\}\subset e} H_e, \quad \text{for $i\neq j$, and}\quad
A_{ii}=0.
\]
We then write \((A,S)\sim \HPC(n,d,k,p)\). For each \(k\), let \(\P_k\) denote the marginal law of \(A\) under \(\HPC(n,d,k,p)\), that is, after averaging over the uniform choice of \(S\).
\end{definition}

\begin{remark}
When $k=0$, the model reduces to the Erd\H{o}s--R\'enyi $d$-uniform hypergraph in which each \emph{$d$-set} $e\in\binom{[n]}{d}$ is present independently with probability $p$, and we denote this null model by
$\ER(n,d,p) \coloneqq \HPC(n,d,0,p)$, with corresponding law $\P_0$.
Moreover, $\P_k=\P_0$ for every $0\le k<d$ as there are no $d$-sets contained in $S$ given $|S| < d$. 
\end{remark}
In this work, we consider the planted clique detection and recovery problem given only the adjacency matrix projection of the hypergraph.

\begin{itemize}
    \item \textbf{Detection:} The detection problem is formulated as a composite hypothesis testing problem:
    \begin{equation}
    \mathcal{H}_0: k=0 \quad \text{vs.} \quad \mathcal{H}_1: k \ge k_0,
    \label{eq:hyp_test_prob}
    \end{equation}
    where $k_0$ represents the critical threshold for the distinguishability of the two hypotheses based on the observed adjacency matrix $A$ of the hypergraph. Formally, let \(\mathcal A\) denote the collection of all \(n\times n\) symmetric matrices with zero diagonal that are obtained as adjacency matrices of \(d\)-uniform hypergraphs on \([n]\). A sequence of test functions $\{\varphi_n\}_{n \ge 1}$, where $\varphi_n \colon \mathcal{A} \to \{0, 1\}$, is said to be \emph{asymptotically powerful} if the total risk $R(\varphi_n)$ vanishes as $n \to \infty$:
    \begin{equation}
    R(\varphi_n) \coloneqq \mathbb{P}_0(\varphi_n(A)=1) + \sup_{k\ge k_0} \mathbb{P}_k(\varphi_n(A)=0) \xrightarrow[n \to \infty]{} 0.
    \label{eq:risk}
    \end{equation}
    The objective is to characterize the regime of $k_0 = k_0(n)$ for which such a sequence of tests exists.

    \item \textbf{Recovery:} The recovery problem is the corresponding estimation task, typically considered in the regime where the presence of a planted structure is guaranteed (i.e., under $\mathcal{H}_1$). Given an observation of the adjacency matrix $A$ sampled from the distribution $\text{HPC}(n, d, k, p)$, where the clique size $k$ is a known parameter, the goal is to identify the latent vertex set $S \subset [n]$. 
    Formally, we seek a sequence of estimators $\{\hat{S}_n\}_{n \ge 1}$, where $\hat{S}_n \colon \mathcal{A} \to \binom{[n]}{k}$ is a function of the adjacency matrix. We say that an estimator achieves \emph{exact recovery} if the probability of misclassification vanishes asymptotically:
    \begin{equation}
    \label{eq:exact_recovery}
    \lim_{n \to \infty} \mathbb{P}_k(\hat{S}_n = S) = 1.
    \end{equation}
    The recovery problem aims to determine the information-theoretic and computational boundaries on $k$ that allow for the consistent identification of the planted support $S$.
\end{itemize}
The paper is organized as follows: Section~\ref{sec:related-work} reviews prior work on planted structures in hypergraphs and related higher-order models. Section~\ref{sec:main} states our main results on detection, exact recovery, and the sparse regime. Section~\ref{sec:discussion} discusses the proof strategy and compares our guarantees with tensor-based benchmarks. Section~\ref{sec:proofs} contains the proofs of the main theorems: Section~\ref{sec:proof_detection} proves the detection result, and Section~\ref{sec:proof_recovery} proves the recovery result. Appendix~\ref{sec:concentration_inequalities} collects the perturbation and concentration tools used throughout the paper. We end this section with a summary of the notation used in the rest of the article.

\subsection{Our contributions}

In this work, we study the planted clique problem under an observation model in which \emph{only} the hypergraph adjacency matrix is available. In contrast, existing recovery results for planted hypergraph clique or dense-subgraph problems typically assume access to the full hypergraph, or equivalently its order-\(d\) adjacency tensor. These works are discussed in Section~\ref{sec:related-work}. As discussed in the start of this section, the adjacency matrix projection loses information relative to the full order-\(d\) adjacency tensor, which makes the present setting more difficult. Our main contributions are:
\begin{itemize}
    \item Detection: We propose a test based on the spectral norm of the adjacency matrix \(A\) and show that it is asymptotically powerful when
    \[
    k_0 \;\gtrsim\; \left(\frac{p}{(1-p)^2}\right)^{\frac1{2(d-1)}}\sqrt n.
    \]
    \item Recovery: We propose a polynomial-time spectral method based on the leading eigenvector of \(A\) and prove exact recovery when
    \[
    k \;\gg\; \left(\frac{p}{1-p}\right)^{\frac1{2(d-1)}}\sqrt n.
    \]
    \item Sparse regime: Both results continue to hold when \(p=p_n\), under explicit sparsity assumptions; in particular, \(p_n\gtrsim n^{-(d-1)}\log n\) for detection and \(p_n\gtrsim n^{-(d-1)}\log^c n\) for recovery.
\end{itemize}

\paragraph{Notation.}
Write $[n]\coloneqq\{1,\dots,n\}$. For a set $e$, $|e|$ denotes its cardinality. A $d$-set is a set of size $d$, and $\binom{S}{d}$ denotes the family of all $d$-subsets of $S\subseteq[n]$.

All asymptotics are as $n\to\infty$. For positive sequences $(a_n)$ and $(b_n)$, we write $a_n=o(b_n)$ if $a_n/b_n\to0$, $a_n=O(b_n)$ if $\limsup_{n\to\infty}a_n/b_n<\infty$, $a_n=\Omega(b_n)$ if $b_n=O(a_n)$, and $a_n\asymp b_n$ if both $a_n=O(b_n)$ and $a_n=\Omega(b_n)$ hold. We also write $a_n\gg b_n$ for $b_n=o(a_n)$, and $a_n \gtrsim b_n$ for $b_n=O(a_n)$.
For $u,v\in\R^n$, write $\langle u, v\rangle = u^\top v$, $\|v\|_2\coloneqq(\sum_{i=1}^n v_i^2)^{1/2}$ and $\|v\|_\infty\coloneqq \max_{i\in[n]}|v_i|$. For $M\in\R^{n\times n}$, let $M_{i:}\in\R^{1\times n}$ denote its $i$th row, and write 
\[
\|M_{i:}\|_2\coloneqq\left(\sum_{j=1}^n M_{ij}^2\right)^{1/2}.
\] 
We write $\|M\|\coloneqq \sup_{\|x\|_2=1}\|Mx\|_2$ for the spectral norm and
\[
\|M\|_{2\to\infty}\coloneqq \sup_{\|x\|_2=1}\|Mx\|_\infty=\max_{i\in[n]}\|M_{i:}\|_2
\]
for the $2\to\infty$ norm. The trace of a matrix $M$ is denoted by $\tr(M)$.

Let $\1_S\in\{0,1\}^n$ be the indicator vector of $S\subset[n]$, and let $J$ and $I$ denote the all-ones and identity matrices of the appropriate dimension. Finally, for symmetric $A\in\R^{n\times n}$, the notation $A\succeq0$ means that $A$ is positive semidefinite, i.e., $x^\top A x\ge0$ for all $x\in\R^n$.

\section{Related work }
\label{sec:related-work}

The hypergraph planted clique problem is a natural extension of the planted clique (PC) and planted dense subgraph (PDS) problems on graphs. PC and PDS have been investigated in \cite{Alon1998Finding,chen2022metropolis,gamarnik2024landscape,hajek2017information,hirahara2024pcEquivalence,mardia2024lowdegree} and the references therein. More recently, PC has also been studied beyond the classical Erd\H{o}s--R\'enyi graph setting, for example in weighted graphs and random geometric graphs \cite{Chatterjee2025Detecting,avrachenkov2025planted}. In this section, we restrict the literature review to works that address planted structures in hypergraphs and closely related higher-order models. In the hypergraphic planted dense subgraph problem (HPDS), one selects a planted vertex set and then samples hyperedges with higher probability within the set than outside; planted cliques correspond to the special case in which all hyperedges within the set are present.

The problem of characterizing the size of the largest clique in a random hypergraph establishes an information-theoretic baseline for detection. In the foundational work \cite{bollobasCliquesRandomGraphs1976}, the authors analyze the clique number \(\omega(H)\) of an Erd\H{o}s--R\'enyi \(d\)-uniform hypergraph \(H\sim \ER(n,d,p)\), and show that \(\omega(H)\) concentrates around
\[
(1+o(1))\left(\frac{d!\log n }{\log (1/p)}\right)^{1/(d-1)}.
\]
More recently, \cite{yuan_shang_2022_detection} establish that if the planted clique size is below this scale, detection is impossible even information-theoretically. Moreover, they give sharp information-theoretic thresholds in the HPDS setting when the edge probabilities are unknown. In a similar direction, \cite{corinzia2022statistical} studies a HPDS problem in a weighted/Gaussian tensor model and determines an information-theoretic threshold for exact recovery as a function of a planted subset size-normalized ratio. Taking the algorithmic perspective, \cite{guruswamiBypassingXORTrick2022} studies polynomial-time spectral/SDP methods for certifying upper bounds on the clique number in Erd\H{o}s--R\'enyi \(d\)-uniform hypergraphs. For \(H\sim \ER(n,d,p)\) with fixed \(d\) and constant \(p\in(0,1)\), they give a certificate \(\omega_{\mathrm{alg}}(H)\) satisfying \(\omega(H)\le \omega_{\mathrm{alg}}(H)\) for every input and
\[
\omega_{\mathrm{alg}}(H)\;\lesssim\; \sqrt{n} \log^{c_d} n
\]
with high probability. Such certification results provide a natural algorithmic baseline for detection under the full-hypergraph observation model.

A related work \cite{zhangTensorSVDStatistical2020} considers a hypothesis testing problem in which the planted clique is assumed to lie in one of two prescribed halves of the vertex set. They observed that detection using a tensor SVD-based approach succeeds whenever $k = \Omega(\sqrt{n})$. Further evidence for the \(\sqrt n\) scale comes from \cite{luo2022tensor}, which shows failure of the low-degree polynomial method for planted hyperclique detection when \(k=o(\sqrt n)\). A matching lower bound for the Sum-of-Squares/Lasserre hierarchy, the hypergraph analogue of the basic SDP relaxation in the graph case, remains open \cite{kothari2026improved,LuoZhang2020OpenProblem}.

In the recovery problem, the task is to identify the planted vertex set rather than merely determine its presence. The work \cite{luo2022tensor} provides a broad overview of the statistical and computational limits of planted high-order models, covering both detection and recovery. It also relates the two tasks: \cite[Lemma~1]{luo2022tensor} shows that a recovery lower bound at size \(k\) implies a detection lower bound at an asymptotically equal scale. Concretely, if no randomized polynomial-time algorithm can recover the planted clique under \(\HPC(n,d,1/2)\) with probability at least \(1-1/n\), then no randomized polynomial-time algorithm can distinguish the null from the planted model with planted size reduced by a constant factor with success probability at least \(1-1/(4n^d)\).

In the hypergraphic PDS framework, \cite[Algorithm~6]{luo2022tensor} proposes a polynomial-time mode-wise aggregation scheme for exact recovery, and \cite[Proposition~1]{luo2022tensor} provides sufficient conditions under which that procedure succeeds. As the closest point of comparison to our recovery result under full tensor observations, we briefly outline their adjacency-tensor-based approach. The method first partitions the vertices into disjoint blocks and then uses these blocks to construct a rescaled subtensor, such that each tensor mode aligns with a distinct block. It then contracts along the remaining \((d-2)\) modes to obtain a collection of two-mode matrices, applies an SVD to each, and uses the leading singular vectors to extract the vertices most aligned with the planted structure within each block. Finally, the vertex sets recovered from all subtensors are combined to obtain the estimated planted set. In particular, \cite[Proposition~1]{luo2022tensor} gives sufficient conditions for exact recovery in the hypergraphic planted dense subgraph model, observed through the full adjacency tensor. When specialized to the planted clique setting, these conditions imply exact recovery at the canonical \(\sqrt n\) scale, namely
\begin{equation}
k \;\gtrsim\; \left(\frac{p}{1-p}\right)^{\frac{1}{2(d-1)}}\sqrt{n}.
\label{eq:recovery_benchmark}
\end{equation}
They further conjecture that this \(\sqrt n\)-scale boundary is optimal for polynomial-time exact recovery \cite[Conjecture~2]{luo2022tensor}.

At the same time, work on hypergraphic PDS in the sparse (log-density) regime points to a separation between detection and exact recovery in higher-order settings. In particular, \cite{dhawan2025dense} studies detection of a planted subhypergraph in a sparse Bernoulli \(d\)-uniform model and uses the low-degree method to characterize regimes of polynomial-time success and failure. Their results support the view that detection can be feasible at weaker signal levels than exact recovery.

\section{Main results}
\label{sec:main}

This section states our main guarantees for the detection and exact recovery problems in the hypergraph planted clique model. Both results are formulated in terms of the centered adjacency matrix
\[
M \;\coloneqq\; A-\E_0[A],
\]
where $A$ is the adjacency matrix associated with the observed hypergraph, constructed as in \eqref{eq:def_adjacency_matrix}, and \(\E_0[A]\) is its expectation under the null model, namely the symmetric matrix with zero diagonal and constant off-diagonal entries \(\binom{n-2}{d-2}p\).

\subsection{Detection}

We first consider the testing problem with the test statistic being the operator norm of $M$. The next theorem gives a sufficient condition under which a spectral norm threshold test is asymptotically powerful.

\begin{theorem}\label{thm:detection}
Fix $d\ge 3$ and $p\in(0,1)$. There exists a constant $C_d>0$ such that the sequence of tests
$\{\varphi_n\}_{n\ge1}$ defined by
\begin{equation}\label{eq:spectral_test}
\varphi_n(A)
\;\coloneqq\;
\1\left\{\|A-\E_0[A]\| > C_d \sqrt{ n\binom{n-2}{d-2}p }\right\}
\end{equation}
is asymptotically powerful for the testing problem \eqref{eq:hyp_test_prob} whenever
\begin{equation}\label{eq:k0_def_clean}
k_0(n) \;>\; (2C_d)^{\frac1{d-1}}
\left(\frac{p}{(1-p)^2}\right)^{\frac{1}{2(d-1)}}\sqrt n,
\end{equation}
\end{theorem}

Theorem~\ref{thm:detection} provides a sufficient condition at the canonical $\sqrt n$ scale, with explicit dependence on $p$ given by \eqref{eq:k0_def_clean}.

\subsection{Exact recovery}

We now turn to the recovery problem. The estimator is spectral: it computes a leading eigenvector of the centered adjacency matrix and selects the $k$ coordinates with the largest magnitude. 

\begin{algorithm}[H]
\small
\algrenewcommand\algorithmicrequire{\textbf{Input:}}
\algrenewcommand\algorithmicensure{\textbf{Output:}}
\caption{Spectral method for planted-set recovery}
\begin{algorithmic}[1]
\Require Adjacency matrix $A$, set size $k$, edge size $d$
\Ensure Estimate $\hat S\subset[n]$ with $|\hat S|=k$
\State Let $u$ be a unit eigenvector corresponding to the largest eigenvalue of $A-\E_0[A]$.
\State Output $\hat S$ as the indices of the $k$ largest entries of $(|u_i|)_{i=1}^n$ (ties broken arbitrarily).
\end{algorithmic}
\label{alg:spectral_recovery_sim}
\end{algorithm}

The next theorem provides a sufficient condition for exact recovery in the sense of \eqref{eq:exact_recovery} under the adjacency matrix observation model.

\begin{theorem}[Spectral recovery]\label{thm:recovery_sim}
Fix $d\ge 3$ and $0<p<1$.
Let $(A,S)$ be sampled from \(\HPC(n,d,k(n),p)\). If
\begin{equation}\label{eq:recovery_thm_scaling}
k(n)\;\gg\;
\left(\frac{p}{1-p}\right)^{\frac{1}{2(d-1)}}\sqrt n.
\end{equation}
then Algorithm~\ref{alg:spectral_recovery_sim} with inputs $A$, $k(n)$ and $d$ achieves exact recovery.
\end{theorem}

\subsection{Sparse regime}
\label{sec:sparse_regime}

Our results for detection and recovery also extend to the case where \(p=p_n\), allowing sparse hypergraphs.

\begin{corollary}[Sparse detection]
\label{cor:pn_extensions_detection}
Let \(p_n>0\) satisfy
\[
p_n\gtrsim n^{-(d-1)}\log n
\qquad\text{and}\qquad
\sup_n p_n<1.
\]
Then the conclusion of Theorem~\ref{thm:detection} remains valid with \(p=p_n\).
\end{corollary}

\begin{corollary}[Sparse recovery]
\label{cor:pn_extensions_recovery}
Let \(p_n>0\) satisfy
\[
p_n\gtrsim n^{-(d-1)}\log^{c_d} n
\qquad\text{and}\qquad
\sup_n p_n<1,
\]
for any constant  \(c_d\ge 4(d-1)/(2d-3)\). Then the conclusion of
Theorem~\ref{thm:recovery_sim} remains valid with \(p=p_n\).
\end{corollary}
\section{Discussion}
\label{sec:discussion}
\subsection{Proof strategy}

The proof of Theorem~\ref{thm:detection} is based on a coupling reduction and a one-dimensional test statistic. Under \(\mathcal H_0\), the operator norm \(\|M\|\) is controlled directly by a concentration bound for the centered adjacency matrix. Under \(\mathcal H_1\), for any planted set \(S\) with \(|S|=k\ge k_0\), we choose a subset \(T\subset S\) of size \(k_0\) and couple the corresponding models so that the hyperedge indicators under \(S\) dominate those under \(T\) coordinatewise. Bounding the operator norm  \(\|M\|\) by the quadratic form \(\langle u_T, Mu_T\rangle\), where \(u_T=\1_T/\sqrt{k_0}\), reduces the type II error analysis to a lower-tail bound for a scalar statistic under the reduced model with clique size \(k_0\). This statistic admits a signal--noise decomposition, in which the signal is deterministic and of order \(k_0^{d-1}\), while the fluctuation is controlled by Bernstein's inequality. This gives the stated \(\sqrt n\)-scale detection threshold.

The proof of Theorem~\ref{thm:recovery_sim} follows the entrywise eigenvector framework of \cite{abbe2020entrywise,gaudio2023community}. Let $M^*$ denote the conditional expectation of $M$ given the planted set $S$, and let $(\lambda^*,u^*)$ denote its leading eigenpair, where
\(
u^*= \1_S/\sqrt{k}.
\)
Exact recovery through \(u\) follows once
\[
2\|u-u^*\|_\infty<k^{-1/2},
\]
since \(u_i^*=k^{-1/2}\) for \(i\in S\) and \(u_i^*=0\) for \(i\notin S\). When only the adjacency matrix is available, one observes pairwise co-occurrence counts rather than hyperedge indicators. Each hyperedge contributes to many pairs \((i,j)\), inducing dependence across the entries and making sharp entrywise perturbation bounds harder to obtain. To handle this, we compare the empirical leading eigenvector \(u\) of \(M\) to the one-step proxy
\[
\frac{Mu^*}{\lambda^*}.
\]
This is natural because \(M^*u^*=\lambda^*u^*\), so the proxy admits the decomposition
\[
\frac{Mu^*}{\lambda^*}
=
u^*+\frac{(M-M^*)u^*}{\lambda^*},
\]
which separates the planted term \(u^*\) from the fluctuation \((M-M^*)u^*/\lambda^*\). To control the entrywise difference between \(u\) and this proxy, we use a leave--one--out construction. For each vertex \(m\), we remove from \(M\) the contribution of all hyperedges incident to \(m\) and study the leading eigenvector of the resulting matrix. Then the \(m\)th row fluctuation of \(M\) depends only on hyperedges incident to \(m\), whereas the leave--one--out eigenvector depends only on the remaining hyperedges. This restores the conditional independence needed for rowwise Bernstein bounds and is the key device that makes the entrywise analysis work under the adjacency-matrix observation model.

\subsection{Comparison with tensor-based benchmarks and outlook}

Our sufficient condition for detection in Theorem~\ref{thm:detection} has the canonical $\sqrt n$ scaling, with explicit $p$-dependence through \eqref{eq:k0_def_clean}. In the clique specialization of the hypergraphic planted dense subhypergraph model, \cite[Conjecture~2]{luo2022tensor} conjectures that the computational transition occurs at the scale given by \eqref{eq:recovery_benchmark}. Using the reduction in \cite[Lemma~1]{luo2022tensor}, this also suggests a detection threshold at \eqref{eq:recovery_benchmark}. When $1-p$ is bounded away from zero, the dependence on $n$ and $p$ in \eqref{eq:k0_def_clean} matches this benchmark up to $d$-dependent constants. This comparison rests on conjectured hardness, and tight computational lower bounds for planted detection remain open even in the classical graph planted clique problem \cite{gamarnik2024landscape}.

Our sufficient condition for recovery in Theorem~\ref{thm:recovery_sim} matches the tensor benchmark \eqref{eq:recovery_benchmark} in its dependence on \(n\) and \(p\). We believe that, with additional careful constant tracking, the same proof would yield the benchmark \eqref{eq:recovery_benchmark} itself; the stronger ''\(\gg\)'' assumption is adopted here for convenience.

Both Theorem~\ref{thm:detection} and Theorem~\ref{thm:recovery_sim} are stated in the oracle setting where \(p\) is known, since the procedures are formulated in terms of the centered matrix \(M=A-\E_0[A]\). To relax this assumption, one could estimate \(p\) from the observed adjacency matrix and plug the estimate into the centering step. Alternatively, one could work directly with \(A\) itself. While \(\1_S\) appears to be an eigenvector of a \(A\) only in degenerate cases with \(p = 0\), \(k = 0\) and \(k = n\), the indicator vectors \(\1_S\) and \(\1_{S^c}\) span the leading two-dimensional eigenspace of \(A\). Thus, one could estimate the top two eigenvectors and cluster the corresponding two-dimensional embedding. We do not pursue this direction here.

Thus, under the adjacency matrix observation model, both detection and recovery admit canonical \(\sqrt n\)-scale sufficient conditions. In the sparse regime, Corollaries~\ref{cor:pn_extensions_detection} and \ref{cor:pn_extensions_recovery} state that detection continues to hold under \(p_n=\Omega(n^{-(d-1)}\log n)\), while our recovery result holds under \(p_n=\Omega(n^{-(d-1)}\log^{c_d} n)\) for a sufficiently large \(c_d\).

A natural next step is to optimize the constants and sparse-regime assumptions, to remove or relax the oracle knowledge of \(p\), and to better understand the limits of inference under the adjacency-matrix observation model.

\section{Proofs}
\label{sec:proofs}
This section collects the proofs of our main results. We first prove Theorem~\ref{thm:detection}. The remaining sections develop the tools needed for Theorem~\ref{thm:recovery_sim}.

\subsection{Proof of Theorem \ref{thm:detection}}
\label{sec:proof_detection}
Write
\begin{equation}
M = A-\E_0[A],
\qquad
t_1 \coloneqq C_{\alpha,d}\sqrt{n\binom{n-2}{d-2}p},
\label{eq:t_1}
\end{equation}
where \(C_{\alpha,d}\) is the constant from Lemma~\ref{lem:lkc_mod}, and define
\[
\kappa(d,p)
\coloneqq
\left(2C_{\alpha,d}\sqrt{(d-2)!}\,\frac{\sqrt p}{1-p}\right)^{\frac1{d-1}}.
\]
Take \(k_0 = k_0(n)\) as the smallest integer strictly larger than \(\kappa(d,p)\sqrt n\), that is
\[
k_0 \coloneqq \lfloor \kappa(d,p)\sqrt n \rfloor + 1.
\] 
Thus, the hypothesis of Theorem \ref{thm:detection} is satisfied, and in particular, \(k_0\to\infty\), and hence \(k_0\ge d\) for all sufficiently large \(n\).

Fix \(\delta\in(0,1)\). It suffices to show that for all sufficiently large \(n\),
\begin{equation}\label{eq:type1_2}
    \P_0(\|M\|>t_1)
+
\sup_{k_0\le k\le n}\P_k(\|M\|\le t_1)
<\delta.
\end{equation}

\subsubsection{Type I error}
\begin{lemma}
\label{lem:type_I_error}
    For $(A,S) \sim \HPC(n,d,k,p)$, the type I error for the spectral test $\varphi_n(\cdot)$ defined in \eqref{eq:spectral_test} satisfies
    \[
    \P_0(\varphi_n(A) =1) \le 4n^{-11}
    \]
\end{lemma}
\begin{proof}
    Using the notation from the beginning of this section and \eqref{eq:type1_2}, \(\P_0(\varphi_n(A) =1) = \P_0(\|M\|>t_1)\). Under \(\cH_0\), \(A-\E_0[A]\) is the centered adjacency matrix. Using Corollary~\ref{cor:lkc_mod} with \(S=\emptyset\) for the concentration of the adjacency matrix of a hypergraph directly proves the lemma.
\end{proof}

\subsubsection{Type II error}

Since the alternative hypothesis is a composite hypothesis, in order to bound the supremum in \eqref{eq:risk} we use a coupling technique and show that recovering the smallest admissable clique size of $k_0(n)$ is the hardest. Subsequently, in order to bound the type II error when a clique of size $k_0(n)$ is present, we prove that a corresponding quadratic form exceeds the threshold \(t_1\) from \eqref{eq:t_1} with high probability.

Fix \(k\in(k_0,n)\), fix a deterministic planted set \(S\subset[n]\) with \(|S|=k\), and write
\(\P_S(\cdot)\coloneqq \P_k(\cdot\mid S)\).
Choose any deterministic subset \(T\subset S\) with \(|T|=k_0\).

Let \(\{U_e\}_{e\in\binom{[n]}{d}}\) be i.i.d.\ \(\mathrm{Unif}[0,1]\), and define two coupled hyperedge families
\[
H^{(S)}_e \coloneqq
\begin{cases}
1, & e\subset S,\\
\1\{U_e\le p\}, & e\not\subset S,
\end{cases}
\qquad
H^{(T)}_e \coloneqq
\begin{cases}
1, & e\subset T,\\
\1\{U_e\le p\}, & e\not\subset T.
\end{cases}
\]
It follows that \(H^{(S)}\) has the law \(\HPC(n,d,k,p)\) conditional on \(S\), while \(H^{(T)}\) has the law
\(\HPC(n,d,k_0,p)\) conditional on \(T\). We use the shorthands \(\P_S\) and \(\P_T\), respectively. Since \(T\subset S\), we have
\begin{equation}
H^{(S)}_e \ge H^{(T)}_e
\qquad\text{for every }e\in\binom{[n]}{d}.
\label{eq:coupling_bound}
\end{equation}

Let
\[
u_T \coloneqq \frac{\1_T}{\sqrt{k_0}}.
\]
For any hyperedge configuration \(H=(H_e)_{e\in\binom{[n]}{d}}\), define
\[
A(H)_{ij}\coloneqq \sum_{e:\,\{i,j\}\subset e} H_e \quad (i\neq j),
\qquad
A(H)_{ii}\coloneqq 0,
\qquad
M(H)\coloneqq A(H)-\E_0[A].
\]

Let
\[
q(H;u_T)\coloneqq \langle u_T,M(H)u_T\rangle.
\]
For \(e\in\binom{[n]}{d}\), define
\[
c_e(T)\coloneqq |e\cap T|\left(|e\cap T|-1\right).
\]

\begin{proposition}[Hyperedge decomposition of the quadratic form]
\label{prop:q_decomposition}
For every hyperedge configuration \(H=(H_e)_{e\in\binom{[n]}{d}}\),
\[
q(H;u_T)
=
\frac{1}{k_0}
\sum_{e\in\binom{[n]}{d}} c_e(T)(H_e-p).
\]
\end{proposition}

\begin{proof}
Since \(u_T=k_0^{-1/2}\1_T\), we have
\[
q(H;u_T)
=
\frac{1}{k_0}
\sum_{i,j\in T} M(H)_{ij}.
\]
Because \(A(H)_{ii}=0\) and \((\E_0[A])_{ii}=0\), it follows that \(M(H)_{ii}=0\), and hence
\begin{equation}
q(H;u_T)
=
\frac{1}{k_0}
\sum_{\substack{i,j\in T\\ i\neq j}} M(H)_{ij}.
\label{eq:quadratic_decomp_interm}
\end{equation}
For \(i\neq j\),
\[
M(H)_{ij}
=
A(H)_{ij}-\E_0[A_{ij}]
=
\sum_{e:\,\{i,j\}\subset e}(H_e-p).
\]
Substituting this into \eqref{eq:quadratic_decomp_interm} yields
\[
q(H;u_T)
=
\frac{1}{k_0}
\sum_{\substack{i,j\in T\\ i\neq j}}
\sum_{e:\,\{i,j\}\subset e}(H_e-p).
\]
Since all sums are finite, we may interchange the order of summation:
\[
q(H;u_T)
=
\frac{1}{k_0}
\sum_{e\in\binom{[n]}{d}}
\left|\left\{(i,j)\in T^2:\ i\neq j,\ \{i,j\}\subset e\right\}\right|
(H_e-p).\]
If \(r=|e\cap T|\), then the number of ordered pairs \((i,j)\in T^2\) with \(i\neq j\) and \(\{i,j\}\subset e\) is exactly \(r(r-1)\). Hence
\[
\left|\left\{(i,j)\in T^2:\ i\neq j,\ \{i,j\}\subset e\right\}\right|
=
|e\cap T|\left(|e\cap T|-1\right)
=
c_e(T),
\]
which gives the claim.
\end{proof}

\begin{lemma}[Coupling lemma]
\begin{equation}
\P_S(\|M\|\le t_1)\wle \P_T(q(H^{(T)}; u_T)\le t_1).
\label{eq:coupling}
\end{equation}
\end{lemma}
\begin{proof}
Since for every \(e\), we have \(c_e(T)\ge 0\) and by \eqref{eq:coupling_bound}, $H^{(S)}_e \ge H^{(T)}_e$ almost surely, it follows that the decomposition given by Proposition \ref{prop:q_decomposition} gives
\begin{equation}
q(H^{(S)}; u_T) = \frac{1}{k_0}
\sum_{e\in\binom{[n]}{d}} c_e(T)(H^{(S)}_e-p) \ge \frac{1}{k_0}
\sum_{e\in\binom{[n]}{d}} c_e(T)(H^{(T)}_e-p) = q(H^{(T)}; u_T)
\label{eq:q_comp}
\end{equation}
almost surely. By definition, $M = M(H^{(S)})$, so that
\begin{equation}
\|M\|
=
\sup_{\|x\|_2=1}|\langle x,M(H^{(S)})x\rangle|
\ge
|\langle u_T,M(H^{(S)})u_T\rangle|
\ge
\langle u_T,M(H^{(S)})u_T\rangle
=
q(H^{(S)}; u_T).
\label{eq:op_bound_seq}
\end{equation}
Hence, \eqref{eq:q_comp} and \eqref{eq:op_bound_seq} give
\[
\{\|M\|\le t_1\}\subseteq \{q(H^{(S)}; u_T)\le t_1\}\subseteq \{q(H^{(T)}; u_T)\le t_1\}.
\]
This implies the claim.
\end{proof}

Thus the type II error analysis reduces to controlling the scalar statistic \(q(H^{(T)};u_T)\) under the reduced alternative \(\P_T\). We next decompose this quantity into a deterministic signal term and a centered fluctuation term.

We proceed to bound \(\P_T(q(H^{(T)}; u_T)\le t_1)\) uniformly over all deterministic \(T\subset[n]\) with \(|T|=k_0\). Fix a deterministic \(T\subset[n]\) with \(|T|=k_0\).
Write
\[
\cE\coloneqq \binom{[n]}{d},
\qquad
\Ein(T)\coloneqq \binom{T}{d},
\qquad
\Eout(T)\coloneqq \cE\setminus \Ein(T).
\]
\begin{observation}[Signal--noise decomposition]
Recalling that \(H^{(T)}_e = 1\) for \(e \in \Ein(T)\), by Proposition~\ref{prop:q_decomposition}, we have the decomposition

\begin{equation}
q(H^{(T)}; u_T)
=
s(n,k_0)
+
\frac{1}{k_0}\sum_{e\in\Eout(T)} c_e(T)(H^{(T)}_e-p),
\label{eq:signal_noise}
\end{equation}
where
\(
s(n,k_0) \coloneqq \frac{1}{k_0}\sum_{e\in\Ein(T)} c_e(T)(1-p).
\)
For the signal term \(s(n,k_0)\), if \(e\in \Ein(T)\), then \(e\subset T\), so \(c_e(T)=d(d-1)\). Hence
\begin{equation}
s(n,k_0)
=
\frac{1-p}{k_0}\,d(d-1)\binom{k_0}{d}
=
(1-p)(k_0-1)\binom{k_0-2}{d-2},
\label{eq:snk_eval}
\end{equation}
where the last equality uses $d(d-1)\binom{k_0}{d}=k_0(k_0-1)\binom{k_0-2}{d-2}$.
\end{observation}

\begin{lemma}
\label{lem:QT_lower_prob}
With \(\P_T\)-probability at least \(1-\delta/2\),
\[
q(H^{(T)}; u_T)
\ge
s(n,k_0)
-
\frac{1}{k_0}\left[
\sqrt{2V(n,k_0)\log\left(\frac{4}{\delta}\right)}
+\frac{2}{3}U_d\log\left(\frac{4}{\delta}\right)
\right],
\]
where \(U_d\coloneqq (d-1)(d-2)\) and
\(
V(n,k_0) \coloneqq p(1-p)\sum_{e\in\Eout(T)} c_e(T)^2
\).
\end{lemma}
\begin{proof}
For the second term in the right-hand side of \eqref{eq:signal_noise}, define
\[
Y_e \coloneqq c_e(T)(H^{(T)}_e-p), \qquad e\in\Eout(T).
\]
Under $\P_T$, the family \(\{Y_e\}_{e\in\Eout(T)}\) is independent and centered. Since
\(e\not\subset T\) implies \(|e\cap T|\le d-1\), we have
\begin{equation}
|Y_e|\le c_e(T)\le U_d,
\qquad
U_d = (d-1)(d-2).
\label{eq:Ye_l1}
\end{equation}
Moreover,
\[
V(n,k_0) = \sum_{e\in\Eout(T)}\Var_T(Y_e)=
p(1-p)\sum_{e\in\Eout(T)} c_e(T)^2.
\]
Since \(c_e(T)^2\le U_d\,c_e(T)\) by \eqref{eq:Ye_l1}, this becomes
\[
V(n,k_0)\le p(1-p)U_d\sum_{e\in\Eout(T)} c_e(T).
\]

Moreover, since \(c_e(T)\) is the number of ordered pairs \((i,j)\in T^2\) with \(i\neq j\) such that
\(\{i,j\}\subset e\), we may count
\begin{align}
\sum_{e\in\Eout(T)}c_e(T)
&=
\sum_{\substack{i,j\in T\\ i\neq j}}
\left|\{e\in\cE:\ \{i,j\}\subset e,\ e\nsubseteq T\}\right| \nonumber\\
&=
k_0(k_0-1)\left(\binom{n-2}{d-2}-\binom{k_0-2}{d-2}\right)
\le k_0^2\binom{n-2}{d-2}.
\label{eq:m_exact_detection_revised}
\end{align}
Therefore,
\begin{equation}
V(n,k_0)
\le
p(1-p)U_d\sum_{e\in\Eout(T)}c_e(T)
\le
p(1-p)U_d\,k_0^2\binom{n-2}{d-2}.
\label{eq:V_bd_detection_revised}
\end{equation}

Given \(V(n,k_0)\) and \(U_d\) from \eqref{eq:Ye_l1} and \eqref{eq:V_bd_detection_revised}, Lemma \ref{lem:scalar-bernstein} gives, for all $t>0$,
\begin{equation}
\P_T\left(\left|\sum_{e\in\Eout}Y_e\right|\ge t\right)
\le
2\exp\left(-\frac{t^2}{ 2V(n,k_0)+\frac{2}{3}U_d t }\right).
\label{eq:Y_bernstein_bound_init}
\end{equation}

Now fix \(\delta\in(0,1)\) and set
\[
\ell \coloneqq \log\left(\frac{4}{\delta}\right)>0,
\qquad
t \coloneqq \sqrt{2V(n,k_0)\,\ell}+\frac{2}{3}U_d\,\ell.
\]
Applying
Lemma~\ref{lem:bernstein_parameter_ineq} together with \eqref{eq:Y_bernstein_bound_init} gives
\[
\P_T\left(
\left|\sum_{e\in\Eout(T)}Y_e\right|
\ge
\sqrt{2V(n,k_0)\log\left(\frac{4}{\delta}\right)}
+\frac{2}{3}U_d\log\left(\frac{4}{\delta}\right)
\right)
\le \frac{\delta}{2}.
\]
Since, by \eqref{eq:signal_noise},
\[
q(H^{(T)}; u_T)
=
s(n,k_0)+\frac{1}{k_0}\sum_{e\in\Eout(T)}Y_e,
\]
the claim follows.
\end{proof}

Lemma~\ref{lem:QT_lower_prob} shows that \(q(H^{(T)};u_T)\) is bounded from below by a deterministic signal term minus two fluctuation terms. It remains to compare these three terms with the threshold \(t_1\). The next proposition verifies that, for all sufficiently large \(n\), the signal dominates both fluctuation terms.

\begin{proposition}
\label{prop:QT_exceeds_t1}
For all sufficiently large \(n\),
\[
s(n,k_0)
-
\frac{1}{k_0}\left[
\sqrt{2V(n,k_0)\ell}+\frac{2}{3}U_d\,\ell
\right]
\;>\; t_1.
\]
\end{proposition}
\begin{proof}
Since
\(
k_0 = \lfloor \kappa(d,p)\sqrt n \rfloor + 1,
\)
we have \(k_0=(1+o(1))\kappa(d,p)\sqrt n\). Since, by \eqref{eq:snk_eval},
\(
s(n,k_0)
=
(1+o(1))\frac{1-p}{(d-2)!}\,k_0^{d-1},
\)
we obtain
\[
s(n,k_0)
\ge
(1+o(1))\frac{1-p}{(d-2)!}\,\kappa(d,p)^{d-1}n^{\frac{d-1}{2}}.
\]
On the other hand,
\begin{equation}
\label{eq:t1_asymp}
t_1
=
C_{\alpha,d}\sqrt{n\binom{n-2}{d-2}p}
=
(1+o(1))C_{\alpha,d}\sqrt{\frac{p}{(d-2)!}}\,n^{\frac{d-1}{2}}.
\end{equation}
Recalling that
\[
\kappa(d,p)^{d-1}
=
2C_{\alpha,d}\sqrt{(d-2)!}\,\frac{\sqrt p}{1-p},
\]
we conclude that
\begin{equation}
s(n,k_0)\ge (2+o(1))\,t_1.
\label{eq:signal_2t1_detection_T}
\end{equation}

Next, using \eqref{eq:V_bd_detection_revised},
\begin{align}
\frac{1}{k_0}\sqrt{2V(n,k_0)\log\left(\frac{4}{\delta}\right)}
&\le
\sqrt{2p(1-p)U_d\binom{n-2}{d-2}\log\left(\frac{4}{\delta}\right)}
\nonumber\\
&=
\frac{\sqrt{2(1-p)U_d\log(4/\delta)}}{C_{\alpha,d}\sqrt n}\,t_1
=
o(t_1),
\label{eq:noise_small_detection_T}
\end{align}
where the second equality uses \eqref{eq:t1_asymp}. Moreover, the linear fluctuation term in Lemma~\ref{lem:QT_lower_prob} bound satisfies
\[
\frac{1}{k_0}\cdot\frac{2}{3}U_d\log\left(\frac{4}{\delta}\right)
\wle
\frac{2U_d}{3k_0}\log\left(\frac{4}{\delta}\right)
=
O\left(\frac{\log(1/\delta)}{\sqrt{n}}\left(\frac{1-p}{\sqrt p}\right)^{\frac1{d-1}}\right).
\]
To compare with $t_1$, by \eqref{eq:t1_asymp}, $t_1 \asymp \sqrt{p}\,n^{\frac{d-1}{2}}$ (for fixed $d$), so
\begin{equation}
\frac{1}{t_1}\cdot
\frac{2U_d}{3k_0}\log\left(\frac{4}{\delta}\right)
=
O\left(
\log(1/\delta)\,
\left(\frac{1-p}{\sqrt p}\right)^{\frac1{d-1}}
\frac{1}{\sqrt p\,n^{\frac d2}}
\right)
=o(1),
\label{eq:linear_small_detection_T}
\end{equation}
and therefore the linear fluctuation term is $o(t_1)$.

Combining \eqref{eq:signal_2t1_detection_T}, \eqref{eq:noise_small_detection_T}, and
\eqref{eq:linear_small_detection_T}, the claim follows.
\end{proof}

Combining this deterministic comparison with Lemma~\ref{lem:QT_lower_prob} yields the desired uniform type II bound.
\begin{lemma}
\label{lem:type_II_error}
For all sufficiently large \(n\),
\[
\sup_{k_0\le k\le n}\P_k(\|M\|\le t_1)\le \frac{\delta}{2}.
\]
\end{lemma}
\begin{proof}   
Together, Lemma~\ref{lem:QT_lower_prob} and Proposition~\ref{prop:QT_exceeds_t1} imply, for all sufficiently large \(n\),
\[
q(H^{(T)}; u_T)
\;>\; t_1
\]
with \(\P_T\)-probability at least \(1-\delta/2\). Consequently,
\begin{equation}
\P_T(q(H^{(T)}; u_T)\le t_1)\le \frac{\delta}{2}
\label{eq:t1_comp}
\end{equation}
for every deterministic \(T\subset[n]\) with \(|T|=k_0\), provided \(n\) is large enough.

Recalling the coupling reduction \eqref{eq:coupling}, for every planted set \(S\) with \(|S|=k\ge k_0\), the bound \eqref{eq:t1_comp} gives
\[
\P_S(\|M\|\le t_1)\le \P_T(q(H^{(T)}; u_T)\le t_1)
 \le \frac{\delta}{2},
\]
and averaging over \(S\) gives the claim.
\end{proof}

\subsubsection{Conclusion of the proof}
By Lemmas \ref{lem:type_I_error} and \ref{lem:type_II_error}, for all sufficiently large \(n\),
\[
\P_0(\|M\|>t_1)\le \frac{\delta}{2}
\qquad\text{and}\qquad
\sup_{k_0\le k\le n}\P_k(\|M\|\le t_1)\le \frac{\delta}{2}.
\]
Hence
\[
\P_0(\|M\|>t_1)+\sup_{k_0\le k\le n}\P_k(\|M\|\le t_1)<\delta.
\]
Since \(\delta\in(0,1)\) was arbitrary, this proves that \(\{\varphi_n\}_{n\ge1}\) is asymptotically powerful.
\qed

\begin{proof}[Proof of Corollary~\ref{cor:pn_extensions_detection}]
Let \(p=p_n\), and define
\[
\kappa_n
\coloneqq
\left(2C_{\alpha,d}\sqrt{(d-2)!}\,\frac{\sqrt{p_n}}{1-p_n}\right)^{\frac1{d-1}},
\quad
k_0\coloneqq \lfloor \kappa_n\sqrt n \rfloor + 1,
\quad
t_1(n)\coloneqq C_{\alpha,d}\sqrt{n\binom{n-2}{d-2}p_n}.
\]

The proof of Theorem~\ref{thm:detection} applies verbatim after replacing \(p\) by \(p_n\), once the three estimates used there are re-checked.

First, since
\[
n\binom{n-2}{d-2}\asymp \frac{n^{d-1}}{(d-2)!},
\]
the assumption \(p_n=\Omega(n^{-(d-1)}\log n)\) implies \eqref{eq:alpha_cond} for all sufficiently large \(n\). Hence the type-I bound from Corollary~\ref{cor:lkc_mod} remains valid.

Second, \(k_0\to\infty\) and \(k_0=(1+o(1))\kappa_n\sqrt n\). Therefore, exactly as in the proof of Theorem~\ref{thm:detection},
\[
s(n,k_0)
=
(1-p_n)(k_0-1)\binom{k_0-2}{d-2}
\ge (2+o(1))\,t_1(n).
\]

Third, the square-root fluctuation term from Lemma~\eqref{lem:QT_lower_prob} bound satisfies, as before,
\[
\frac{1}{k_0}\sqrt{2V(n,k_0)\log\left(\frac{4}{\delta}\right)}
\le
\frac{\sqrt{2(1-p_n)U_d\log(4/\delta)}}{C_{\alpha,d}\sqrt n}\,t_1(n)
=
o\left(t_1(n)\right).
\]
For the linear term in Lemma~\eqref{lem:QT_lower_prob} bound, note that
\[
k_0=\Omega\left((\log n)^{1/(2(d-1))}\right)
\qquad\text{and}\qquad
t_1(n)=\Omega\left(\sqrt{\log n}\right),
\]
the latter because \(n\binom{n-2}{d-2}p_n=\Omega(\log n)\). Hence
\[
\frac{1}{t_1(n)}\cdot \frac{2U_d}{3k_0}\log\left(\frac{4}{\delta}\right)
=
o(1).
\]
Thus the linear term in the \(p_n\)-analogue is also \(o(t_1(n))\), and a version of Proposition~\ref{prop:QT_exceeds_t1} follows. The rest of the argument is identical to the proof of Theorem~\ref{thm:detection}, and yields asymptotic power.
\end{proof}

\subsection{Proof of Theorem \ref{thm:recovery_sim}}
\label{sec:proof_recovery}

Fix a planted set $S\subset[n]$ of size $k = k(n)$. Throughout this section, we work under the conditional law
$\P_S(\cdot)= \P(\cdot\mid S)$.  The expectation and variance with respect to $\P_S$ are denoted using $\E_S$ and $\Var_S$. Set
\[
M = A-\E_0[A],
\qquad
M^* \coloneqq \E_S[M].
\]
All generic perturbation and concentration tools used below are recorded in Appendix~\ref{sec:concentration_inequalities}. We assume throughout that the hypothesis \eqref{eq:recovery_thm_scaling} of Theorem~\ref{thm:recovery_sim} holds.

\subsubsection{Spectral analysis}
\label{sec:spectrum}

Let $\cE$ denote the collection of all $d$-subsets of $[n]$ (hyperedges). Define
\[
\Ein \coloneqq \{e\in\cE: e\subset S\},
\qquad
\Eout \coloneqq \cE\setminus\Ein,
\qquad
\Eout_i \coloneqq \{e\in\Eout: i\in e\}.
\]
Under $\P_S$, $\{H_e\}_{e\in\Eout}$ is a collection of independent random variables $H_e\sim\Ber(p)$, while for $e\in\Ein$, $H_e = 1$ almost surely.
Thus, for $i\neq j$,
\[
\E_S[A]_{ij}= \sum_{e\ni i,j}\E_S[H_e] = 
\begin{cases}
w_{\mathrm{in}},& \{i,j\}\subseteq S,\\
w_{\mathrm{out}},& \{i,j\}\not\subseteq S,
\end{cases}
\]
where
\begin{equation}
w_{\mathrm{in}}=\binom{k-2}{d-2}+p\left(\binom{n-2}{d-2}-\binom{k-2}{d-2}\right),
\qquad
w_{\mathrm{out}}=p\binom{n-2}{d-2}.
\label{eq:win_wout}
\end{equation}
Since  $\E_0[A]=w_{\mathrm{out}}(J-I)$, we obtain
\[
M^*= \E_S[M]
=(w_{\mathrm{in}}-w_{\mathrm{out}}) (\1_S\1_S^\top-I_S),
\]
where $I_S = \diag(\1_S)$. 
Let $\lambda\coloneqq \lambda_1(M)$ denote the leading eigenvalue of $M$.
The top eigenpair of $M^*$ is $(\lambda^*,u^*)$ with
\begin{equation} 
u^*=\frac{\1_S}{\sqrt{k}},
\qquad
\lambda^*=(w_{\mathrm{in}}-w_{\mathrm{out}})(k-1) = (1-p)(k-1)\binom{k-2}{d-2},
\label{eq:lambda_star}
\end{equation}
where the last equality is by \eqref{eq:win_wout}.
On the subspace spanned by the standard basis vectors indexed by $S^c$, the eigenvalue is zero with multiplicity $n-k$.
On the subspace of vectors supported on $S$ and orthogonal to $\1_S$, the eigenvalues equal $-(w_{\text{in}}-w_{\text{out}})$ with multiplicity $k-1$. The following lemma tracks the asymptotic scale of $\lambda^*$.

\begin{lemma}[$\lambda^*$-scale]\label{lem:lamstar_scale}
Fix \(d\ge 3\) and \(0<p<1\), and assume \eqref{eq:recovery_thm_scaling}. Then
\begin{equation}
\frac{\sqrt{p\,n^{d-1}}}{\lambda^*} =o(1), \quad
\frac{\sqrt{p\,k\,n^{d-2}\log n}}{\lambda^*}
=o(1), \quad\text{and}\quad
\frac{\sqrt{k} \log n}{\lambda^*}
=o(1).
\label{eq:lamstar_scale}
\end{equation}
Moreover, the same three conclusions hold with \(p\) replaced by \(p_n\), provided
\(\sup_n p_n<1\), \eqref{eq:recovery_thm_scaling} is interpreted with \(p_n\), and
\begin{equation}
\label{eq:p_n_lb}
p_n=\Omega\left(n^{-(d-1)}\log^{c_d} n\right)
\end{equation}
for some \(c_d\ge \dfrac{4(d-1)}{2d-3}\).
\end{lemma}

\begin{proof}
It is enough to prove the \(p_n\)-extension, since the fixed-\(p\) case is obtained by
taking \(p_n\equiv p\), for which the additional assumptions are automatic.

Set
\[
r_n \coloneqq\frac{p_n^{\frac{1}{2(d-1)}}\sqrt n}{k}.
\]
By \eqref{eq:recovery_thm_scaling}, we have \(r_n=o(1)\). Moreover,
\[
p_n^{\frac{1}{2(d-1)}}\sqrt n
=\Omega \left(\log^{\frac{c_d}{2(d-1)}}n\right),
\]
so \(k\to\infty\). Since \(\sup_n p_n<1\), there exists \(\eta>0\) such that
\(1-p_n\ge \eta\) for all \(n\). Hence, by \eqref{eq:lambda_star},
\[
\lambda^*
=(1-p_n)(k-1)\binom{k-2}{d-2}
\asymp k^{d-1}.
\]
Therefore
\[
\frac{\sqrt{p_n\,n^{d-1}}}{\lambda^*}
\asymp
\frac{\sqrt{p_n}\,n^{\frac{d-1}{2}}}{k^{d-1}}
=
r_n^{\,d-1}
=o(1).
\]
Likewise,
\[
\frac{\sqrt{p_n\,k\,n^{d-2}\log n}}{\lambda^*}
\asymp
\frac{\sqrt{p_n}\,n^{\frac{d-2}{2}}\sqrt{\log n}}{k^{d-\frac32}}
=
r_n^{\,d-\frac32}\,
p_n^{\frac{1}{4(d-1)}}\,
\frac{\sqrt{\log n}}{n^{1/4}}
=o(1),
\]
because \(\sup_n p_n<1\) implies
\(p_n^{1/(4(d-1))}=O(1)\), while \(\sqrt{\log n}/n^{1/4} = o(1)\).
Finally,
\[
\frac{\sqrt{k}\,\log n}{\lambda^*}
\asymp
\frac{\log n}{k^{d-\frac32}}
=
r_n^{\,d-\frac32}\,
\frac{\log n}{
p_n^{\frac{2d-3}{4(d-1)}}\,n^{\frac{2d-3}{4}}}.
\]
Using \eqref{eq:p_n_lb},
the denominator on the right-hand side becomes
\[
p_n^{\frac{2d-3}{4(d-1)}}\,n^{\frac{2d-3}{4}}
=
\Omega \left(
\log^{\frac{c_d(2d-3)}{4(d-1)}}n
\right),
\]
Hence
\[
\frac{\sqrt{k}\,\log n}{\lambda^*}
=
o \left(
\log^{\,1-\frac{c_d(2d-3)}{4(d-1)}}n
\right)
=
o(1),
\]
because \(c_d\ge \dfrac{4(d-1)}{2d-3}\).
\end{proof}

Define the event
\begin{equation}
E_1 \coloneqq \left\{\|M-M^*\|\le C_{1,d}\sqrt{n\binom{n-2}{d-2}p}\right\},
\label{eq:M_concentration}
\end{equation}
where \(C_{1,d}\) is the constant from Lemma~\ref{lem:lkc_mod} with \(\alpha = 1\).
\begin{claim}\label{clm:PSE1} For all sufficiently large \(n\),
\begin{equation}
\P_S(E_1)\ge 1-2n^{-10}.
\end{equation}
\end{claim}
\begin{proof}
    This is a direct consequence of Corollary~\ref{cor:lkc_mod} and the fact that \(4n^{-11} \le 2n^{-10}\) for all \(n\ge 2\).
\end{proof}

\begin{lemma}\label{lem:Delta0_vanish}
Fix \(d\ge 3\) and \(0<p<1\). Define
\[
\Delta_0 \;\coloneqq\; \frac{\|M-M^*\|}{\lambda^*},
\]
and assume \eqref{eq:recovery_thm_scaling}. Then, on \(E_1\),
\(
\Delta_0=o(1).
\)
\end{lemma}

\begin{proof}
On \(E_1\), \eqref{eq:M_concentration} gives
\[
\Delta_0
\le
\frac{C_{1,d}\sqrt{p}\,\sqrt{n\binom{n-2}{d-2}}}{\lambda^*}
\le
C_{1,d}\frac{\sqrt{p\,n^{d-1}}}{\lambda^*}
=
o(1),
\]
where we used \(\binom{n-2}{d-2}\le n^{d-2}\) and
Lemma~\ref{lem:lamstar_scale}.
\end{proof}
\subsubsection{Entrywise bound via a one-step proxy}

In this subsection, we derive sufficient conditions that guarantee that Algorithm \ref{alg:spectral_recovery_sim} recovers the planted clique. Let $(\lambda,u)$ denote the leading eigenpair of $M$ and let $(\lambda^*,u^*)$ be as defined in \eqref{eq:lambda_star}.
Define
\begin{equation}
\alpha_n
\coloneqq
\left\|\frac{Mu^*}{\lambda^*}-u^*\right\|_\infty,
\qquad
\beta_n \coloneqq \left\|u-\frac{Mu^*}{\lambda^*}\right\|_\infty.
\label{eq:alpha_beta_def}
\end{equation}

\begin{lemma}
\label{lem:suff_cond1}
For \((A,S)\sim \HPC(n,d,k,p)\), suppose there exists a deterministic sequence
\((\varepsilon_n)_{n\ge1}\) with \(\varepsilon_n\to0\) such that
\[
\P_S\left(\alpha_n+\beta_n \le \frac{\varepsilon_n}{\sqrt{k}}\right)\to 1
\qquad\text{as }n\to\infty.
\]
Then Algorithm~\ref{alg:spectral_recovery_sim} exactly recovers the clique \(S\), i.e.,
\(
\P_k(\hat S=S)\to 1.
\)
\end{lemma}

\begin{proof}
Algorithm~\ref{alg:spectral_recovery_sim} exactly recovers \(S\) if and only if
\[
\min_{i\in S}|u_i|>\max_{j\notin S}|u_j|.
\]
By the triangle inequality,
\[
\|u-u^*\|_\infty\le \alpha_n+\beta_n.
\]
Since \(u^*=\1_S/\sqrt{k}\), we have \(u_i^*=1/\sqrt{k}\) for \(i\in S\) and \(u_j^*=0\) for \(j\notin S\). Hence
\begin{align}
|u_i|
&\ge |u_i^*|-\|u-u^*\|_\infty
\ge \frac1{\sqrt{k}}-(\alpha_n+\beta_n)
\qquad \text{for } i\in S, \nonumber\\
|u_j|
&\le |u_j^*|+\|u-u^*\|_\infty
\le \alpha_n+\beta_n
\qquad \text{for } j\notin S.
\label{eq:u_in_u_out_sep}
\end{align}
Choose \(n\) large enough that \(\varepsilon_n<1/2\). On the event
\[
\left\{\alpha_n+\beta_n \le \frac{\varepsilon_n}{\sqrt{k}}\right\},
\]
it follows from \eqref{eq:u_in_u_out_sep} that
\[
\min_{i\in S}|u_i|
\ge \frac{1-\varepsilon_n}{\sqrt{k}}
>
\frac{\varepsilon_n}{\sqrt{k}}
\ge \max_{j\notin S}|u_j|.
\]
Therefore
\[
\P_S(\hat S=S)
\ge
\P_S\left(\alpha_n+\beta_n \le \frac{\varepsilon_n}{\sqrt{k}}\right)
\to 1.
\]
Averaging over the uniform choice of the clique \(S\) of size \(k\) yields
\[
\P_k(\hat S=S)
=
\E_k\left[\P(\hat S=S\mid S)\right]
=
\E_k\left[\P_S(\hat S=S)\right]
\to 1,
\]
which proves exact recovery.
\end{proof}

It remains to bound both the terms $\alpha_n$ and $\beta_n$. While $\alpha_n$ is easy to bound (see Lemma \ref{lem:pivot-Mu-star}), the term $\beta_n$ requires significant work. To begin, we introduce a few definitions. For each \(m\in[n]\), let \(M^{(-m)}\) be the \emph{leave--one--out} matrix obtained from \(M\) by removing the contribution of all hyperedges incident to \(m\), and set
\[
B^{(m)}\coloneqq M-M^{(-m)}.
\]
Since \(M^{(-m)}\) is symmetric, all its eigenvalues are real. By construction, \(M^{(-m)}_{m:}=\0\), so \(0\) is an eigenvalue of \(M^{(-m)}\). Moreover,
\(M^{(-m)}\) has zero diagonal, hence \(\tr(M^{(-m)})=0\). Therefore, if \(M^{(-m)}\neq \0\), we have \(\lambda_1(M^{(-m)})>0\), because all eigenvalues cannot be nonpositive. Since any eigenvector corresponding to a nonzero eigenvalue
of \(M^{(-m)}\) has zero \(m\)th coordinate, every leading unit eigenvector \(v\) of \(M^{(-m)}\) satisfies
\[
v_m=0.
\]
Fix once and for all a measurable tie-break rule that selects, for each \(m\in[n]\), a leading unit eigenvector \(\bar u^{(-m)}\) of \(M^{(-m)}\) satisfying
\[
\bar u_m^{(-m)}=0.
\]
If \(M^{(-m)}=0\), we use the same measurable tie-breaker among unit vectors satisfying
\(\bar u_m^{(-m)}=0\). We then define \(u^{(-m)}\) by imposing the sign convention from Remark~\ref{rem:sign} on \(\bar u^{(-m)}\). Thus
\[
u^{(-m)} = s_m \bar u^{(-m)}
\qquad\text{for some } s_m\in\{\pm1\}.
\]

\begin{remark}[Sign convention]
\label{rem:sign}
Eigenvectors are defined only up to a global sign. Since Algorithm~\ref{alg:spectral_recovery_sim} depends on the leading eigenvector only through the magnitudes \((|u_i|)_{i\in[n]}\), its output is unchanged if \(u\) is replaced by \(-u\). Accordingly, throughout the proof we choose the signs of \(u\) and \(u^{(-m)}\) so that
\[
\langle u,u^*\rangle \ge 0,
\qquad
\langle u^{(-m)},u\rangle \ge 0
\quad\text{for every }m\in[n],
\]
and on the tie events \(\{\langle u,u^*\rangle=0\}\) or \(\{\langle u^{(-m)},u\rangle=0\}\) we use a measurable tie-breaker.
\end{remark}

For each \(m\in[n]\), define the hyperedge subsets
\begin{equation}
\cE_m\coloneqq \{e\in\tbinom{[n]}{d}: m\in e\},
\qquad
\cE_{m,\mathrm{in}}\coloneqq \{e\in\cE_m: e\subset S\},
\qquad
\cE_{m,\mathrm{out}}\coloneqq \cE_m\setminus \cE_{m,\mathrm{in}},
\label{eq:m-in-def}
\end{equation}
and the generated sigma algebras
\[
\cF^{(m)}\coloneqq \sigma\left(\{H_e: e\in\cE_m\}\right),
\qquad
\cF^{(-m)}\coloneqq \sigma\left(\{H_e: e\notin\cE_m\}\right).
\]
Since \(\bar u^{(-m)}\) is a deterministic function of \(M^{(-m)}\), it is \(\cF^{(-m)}\)-measurable.

\begin{proposition}\label{prop:chain-Mu-star}
Suppose $\Delta_0\le \frac12$, and define
\[
r_n
\coloneqq
\dfrac{1}{\lambda^*}\left(\max_{m\in[n]} \left|M_{m:}(u^{(-m)}-u^*)\right|
+
\|M\|_{2\to\infty}\max_{m\in[n]}\|u-u^{(-m)}\|_2\right).
\]
Then,
\[
\left\| u-\dfrac{Mu^*}{\lambda^*} \right\|_\infty
\le
2\Delta_0\dfrac{\|Mu^*\|_\infty}{\lambda^*}\,+2r_n.
\]
\end{proposition}

\begin{proof}
By Weyl's inequality \eqref{ineq:weyl}, applied with $i = 1$,
\(
|\lambda-\lambda^*|\le \|M-M^*\|=\Delta_0 \lambda^*.
\)
Since $\lambda^*>0$ and $\Delta_0\le\frac12$, we obtain $\lambda\ge (1-\Delta_0)\lambda^*>0.$
Therefore,
\begin{equation}
\left|\frac1\lambda-\frac1{\lambda^*}\right|
=  \frac{|\lambda-\lambda^*|}{|\lambda| \lambda^*}
\wle  \frac{\Delta_0}{(1-\Delta_0)\lambda^*}
\wle  \frac{2\Delta_0}{\lambda^*}.
\label{eq:lambda_bound}
\end{equation}
Using $Mu=\lambda u$, the triangle inequality yields
\[
\left\|u-\frac{Mu^*}{\lambda^*}\right\|_\infty
=
\left\|\frac{Mu}{\lambda}-\frac{Mu^*}{\lambda^*}\right\|_\infty
\le
\left|\frac1\lambda-\frac1{\lambda^*}\right|\|Mu^*\|_\infty+\frac1{|\lambda|}\|M(u-u^*)\|_\infty,
\]
and \eqref{eq:lambda_bound} gives
\begin{equation}
\left\|u-\frac{Mu^*}{\lambda^*}\right\|_\infty
\le
\frac{2\Delta_0}{\lambda^*}\|Mu^*\|_\infty + \frac{2}{\lambda^*}\|M(u-u^*)\|_\infty.
\label{eq:inf_norm_u_one_step}
\end{equation}

Fix $m\in[n]$. Using the leave--one--out vector $u^{(-m)}$,
\[
|[M(u-u^*)]_m|
=|M_{m:}(u-u^*)|
\le |M_{m:}(u^{(-m)}-u^*)|
   +|M_{m:}(u-u^{(-m)})|.
\]
Using the Cauchy--Schwarz inequality \eqref{ineq:cs}, and the definition of the two-to-infinity norm,
\[
|M_{m:}(u-u^{(-m)})|
\le \|M_{m:}\|_2 \|u-u^{(-m)}\|_2
\le \|M\|_{2\to\infty} \|u-u^{(-m)}\|_2.
\]
Taking maxima over $m$ gives
\begin{equation}
\|M(u-u^*)\|_\infty
\le \max_m |M_{m:}(u^{(-m)}-u^*)|
+\|M\|_{2\to\infty} \max_m\|u-u^{(-m)}\|_2
= \lambda^* r_n.
\label{eq:RM_bound}
\end{equation}
Substituting into \eqref{eq:inf_norm_u_one_step} concludes the proof.
\end{proof}

Thus, to bound \(\beta_n\) via Proposition~\ref{prop:chain-Mu-star}, it suffices to control
\begin{equation}
\|Mu^*\|_\infty, \quad
\|M\|_{2\to\infty}, \quad
\max_{m\in[n]} \left|M_{m:}(u^{(-m)}-u^*)\right|, \quad
\max_{m\in[n]}\|u-u^{(-m)}\|_2.
\label{eq:key_quantities}
\end{equation}
The leave--one--out construction helps tackle the dependencies in the last two terms of \eqref{eq:key_quantities}. By construction, \(\bar u^{(-m)}\) is measurable with respect to hyperedges not containing \(m\), whereas \(M_{m:}\) is measurable with respect to hyperedges containing \(m\); under \(\P_S\) these families are independent. The perturbation matrix \(B^{(m)}\) and \(\max_{m\in[n]}\|u-u^{(-m)}\|_2\) are analyzed in Sections~\ref{sec:Bm_u_bound_strategy} and \ref{sec:loo_ev_bound}. These are applied in Section~\ref{sec:row-wise} to obtain a bound on the row concentration term \(\max_{m\in[n]} \left|M_{m:}(u^{(-m)}-u^*)\right|\), first for \(\bar u^{(-m)}\) and then, by the sign argument, for \(u^{(-m)}\). Section~\ref{sec:recovery_conclusion} combines the bounds on each quantity in \eqref{eq:key_quantities} to apply Proposition~\ref{prop:chain-Mu-star}.

\subsubsection{Bounds on the leave--one--out perturbations}
\label{sec:Bm_u_bound_strategy}
We proceed to obtain bounds on the operator norm of the perturbation \(B^{(m)}=M-M^{(-m)}\) and consequently analyze the leading eigenvectors of \(M\) and \(M^{(-m)}\).

\begin{lemma}[Leave--one--out eigenvector comparison]\label{lem:loo-evec-comparison-DK}
Assume that \(\lambda_1(M^{(-m)})\) has multiplicity one and that
\[
\delta_m(\lambda)\coloneqq \min_{i\ge 2}\left| \lambda_i(M^{(-m)})-\lambda \right|>0.
\]
Then
\begin{equation}\label{eq:u-um_l2}
\|u-u^{(-m)}\|_2
\le \frac{\sqrt{2}}{\delta_m(\lambda)} \|B^{(m)}u\|_2,
\end{equation}
where
\begin{equation}
\delta_m(\lambda)
\wge  \lambda^* - 2\|M-M^*\| - \|B^{(m)}\|.
\label{eq:delta_m_def}
\end{equation}
\end{lemma}

\begin{proof}
Using Corollary~\ref{cor:DK_residual_rankone} with
\(A=M^{(-m)}\), \(v=u^{(-m)}\), \(x=u\), and \(\hat{\lambda}= \lambda\), yields
\[
\|u-u^{(-m)}\|_2
\le \frac{\sqrt{2}}{\delta_m(\lambda)} \|(M^{(-m)}-\lambda I)u\|_2.
\]
Since \((M-\lambda I)u=0\) and \(M-M^{(-m)}=B^{(m)}\), we obtain
\eqref{eq:u-um_l2}. It remains to compute \(\delta_m(\lambda)\). Using the reverse triangle inequality and Weyl's inequality \eqref{ineq:weyl}, for each \(i\ge2\),
\begin{align}
\min_{i\ge2} \left| \lambda_i(M^{(-m)})-\lambda \right|
&\wge
\min_{i\ge2} |\lambda_i(M^*)-\lambda|-|\lambda_i(M^{(-m)})-\lambda_i(M^*)|\nonumber\\
&\wge
\min_{i\ge2}|\lambda_i(M^*)-\lambda|-\|M^{(-m)}-M^*\|\nonumber\\
&\wge
\min_{i\ge2}|\lambda_i(M^*)-\lambda^*|-|\lambda-\lambda^*|-\|M^{(-m)}-M^*\|.
\label{eq:delta_chain_loo}
\end{align}
Since \(\lambda^*=\lambda_1(M^*)>0\) and
\(\lambda_i(M^*)\in\{0,\,-(w_{\rm in}-w_{\rm out})\}\) for \(i\ge2\),
\begin{equation}\label{eq:delta_min_struct}
\min_{i\ge2}|\lambda_i(M^*)-\lambda^*|
=
\min\left\{|0-\lambda^*|,\;|-(w_{\rm in}-w_{\rm out})-\lambda^*|\right\}
=\lambda^*.
\end{equation}
Moreover, by Weyl's inequality \eqref{ineq:weyl}, applied with \(i = 1\),
\[
|\lambda-\lambda^*|=|\lambda_1(M)-\lambda_1(M^*)|\wle \|M-M^*\|,
\]
and using the triangle inequality
\begin{equation}\label{eq:delta_Mm_triangle}
\|M^{(-m)}-M^*\|
=\|(M-M^*)-(M-M^{(-m)})\|
\wle \|M-M^*\|+\|B^{(m)}\|.
\end{equation}
Substituting in \eqref{eq:delta_chain_loo} gives \eqref{eq:delta_m_def}.
\end{proof}

\begin{lemma}[Leave--one--out perturbation bounds]
\label{lem:fixed_point_simple_rigorous}
Assume that \(\lambda_1(M^{(-m)})\) has multiplicity one, that \(\delta_m(\lambda)>0\), and define
\[
\eta_m\coloneqq \frac{\sqrt2\,\|B^{(m)}\|}{\delta_m(\lambda)}.
\]
If \(\eta_m<1\), then:
\[
\|B^{(m)}u\|_2 \le \dfrac{1}{1-\eta_m}\,\|B^{(m)}u^{(-m)}\|_2.
\]
\end{lemma}

\begin{proof}
By Lemma~\ref{lem:loo-evec-comparison-DK},
\begin{equation}\label{eq:fp_step_DK}
\|u-u^{(-m)}\|_2
\wle
\frac{\sqrt2}{\delta_m(\lambda)}\,\|B^{(m)}u\|_2.
\end{equation}
Also, by the triangle inequality and \(\|Bx\|_2\le \|B\|\,\|x\|_2\),
\[
\|B^{(m)}u\|_2
\le
\|B^{(m)}u^{(-m)}\|_2+\|B^{(m)}\|\,\|u-u^{(-m)}\|_2.
\]
Substituting \eqref{eq:fp_step_DK} gives
\[
\|B^{(m)}u\|_2
\le
\|B^{(m)}u^{(-m)}\|_2+\eta_m\|B^{(m)}u\|_2.
\]
Hence
\[
(1-\eta_m)\|B^{(m)}u\|_2 \le \|B^{(m)}u^{(-m)}\|_2,
\]
which together with \(\eta_m<1\) yields the claim.
\end{proof}

Notice that Lemma~\ref{lem:loo-evec-comparison-DK} and Lemma~\ref{lem:fixed_point_simple_rigorous} are deterministic statements: they hold on any realization for which their stated assumptions are satisfied. In the following, we obtain bounds on the terms in the lemma statements by using the probabilistic structure of the model. In particular, we decompose $B^{(m)}$ into a deterministic planted-clique component and a centered fluctuation, and then apply Bernstein's inequalities to control the latter. We begin by introducing a few definitions.

For \(e\in\binom{[n]}{d}\), let \(B_e\in\R^{n\times n}\) be the co-membership matrix of the hyperedge \(e\),
\[
(B_e)_{ij}\coloneqq \1\{i\neq j,\ \{i,j\}\subset e\}.
\]

\begin{lemma}[Single--hyperedge bounds]\label{lem:single_edge_action_bounds_pub}
For any \(e\in\binom{[n]}{d}\) and \(v\in\R^n\),
\begin{equation*}
\|B_ev\|_2\le (d-1)\left(\sum_{j\in e}v_j^2\right)^{1/2},
\end{equation*}
\end{lemma}
\begin{proof}
If \(i\notin e\), then \((B_ev)_i=0\). If \(i\in e\), then
\(
(B_ev)_i=\sum_{j\in e\setminus\{i\}}v_j.
\label{eq:single_edge_coordinate_formula}
\)
Hence \(\supp(B_ev)\subseteq e\), and therefore
\begin{equation}
\|B_ev\|_2^2
=
\sum_{i\in e}\left(\sum_{j\in e\setminus\{i\}}v_j\right)^2.
\label{eq:single_edge_l2_expand}
\end{equation}
By the Cauchy--Schwarz inequality \eqref{ineq:cs}, for each \(i\in e\),
\begin{equation}
\left(\sum_{j\in e\setminus\{i\}}v_j\right)^2
\le
(d-1)\sum_{j\in e\setminus\{i\}}v_j^2.
\label{eq:single_edge_cs}
\end{equation}
Substituting \eqref{eq:single_edge_cs} into \eqref{eq:single_edge_l2_expand} gives
\begin{equation}
\|B_ev\|_2^2
\le
(d-1)\sum_{i\in e}\sum_{j\in e\setminus\{i\}}v_j^2.
\label{eq:Bev_init_bound}
\end{equation}
For each fixed \(j\in e\), the term \(v_j^2\) appears once for every \(i\in e\setminus\{j\}\), that is, exactly \(d-1\) times. Hence
\begin{equation}
\sum_{i\in e}\sum_{j\in e\setminus\{i\}}v_j^2
=
(d-1)\sum_{j\in e}v_j^2.
\label{eq:single_edge_double_count}
\end{equation}
Substituting \eqref{eq:single_edge_double_count} into \eqref{eq:Bev_init_bound} concludes the proof.
\end{proof}

Recall that \(\bar u^{(-m)}\) is \(\cF^{(-m)}\)-measurable by construction, while \(u^{(-m)} = s_m\bar u^{(-m)}\) is not, \(s_m \in \{\pm 1\}\) being a function of \(u\) by Remark~\ref{rem:sign}. The following Bernstein arguments control \(u^{(-m)}\) through \(\cF^{(-m)}\)-measurability of \(\bar u^{(-m)}\). For sign-invariant quantities we may replace \(u^{(-m)}\) by \(\bar u^{(-m)}\) without change. For quantities that depend on the sign, we prove the same bound for both \(\bar u^{(-m)}\) and \(-\bar u^{(-m)}\), and therefore also for \(u^{(-m)}\).

\begin{observation}[In/out decomposition]
\label{obs:Bm_inout_decomp_pub}
For each \(m\in[n]\), under \(\P_S\), the matrix \(B^{(m)}\) admits the decomposition
\[
B^{(m)}=P^{(m)}+W^{(m)},
\]
where
\[
P^{(m)}\coloneqq \sum_{e\in\cE_{m,\mathrm{in}}}(1-p)B_e,
\qquad
W^{(m)}\coloneqq \sum_{e\in\cE_{m,\mathrm{out}}}(H_e-p)B_e.
\]
\end{observation}

We begin with the operator-norm control of \(W^{(m)}\). Let \(C_{1,d}\) be the constant in Lemma~\ref{lem:lkc_mod} with \(\alpha = 1\), and define
\begin{equation}
E_{\mathrm{op}}
\coloneqq
\left\{
\max_{m\in[n]}\|W^{(m)}\|
\le
C_{1,d}\,\sqrt{n\binom{n-2}{d-2}\,p}
\right\}.
\label{eq:Eop_def_lkc}
\end{equation}

\begin{claim}
\label{clm:Wm_op_uniform_lkc}
If \(p \ge \alpha\,\dfrac{\log n}{n\binom{n-2}{d-2}}\), then
\[
\P_S(E_{\mathrm{op}})\ge 1-4n^{-10}.
\]
\end{claim}

\begin{proof}
For fixed \(m\in[n]\), under \(\P_S\) the set \(\cE_{m,\mathrm{out}}\) is deterministic and
\(\{H_e\}_{e\in\cE_{m,\mathrm{out}}}\) is a collection of independent \(\Ber(p)\) random variables. Also,
\[
W^{(m)}
=\sum_{e\in\cE_{m,\mathrm{out}}}(H_e-p)\,B_e
=
A(\cE_{m,\mathrm{out}})-\E_S[A(\cE_{m,\mathrm{out}})].
\]
Thus Lemma~\ref{lem:lkc_mod} yields
\[
\|W^{(m)}\|
\le
C_{1,d}\,\sqrt{n\binom{n-2}{d-2}\,p}
\]
with probability at least \(1-4n^{-11}\). A union bound over \(m\in[n]\) gives the claim.
\end{proof}

\begin{corollary}
\label{cor:Wm_op_littleo}
Assume \eqref{eq:recovery_thm_scaling}. On \(E_{\mathrm{op}}\),
\[
\max_{m\in[n]}\|W^{(m)}\|=o(\lambda^*).
\]
\end{corollary}

\begin{proof}
On \(E_{\mathrm{op}}\),
\[
\max_{m\in[n]}\frac{\|W^{(m)}\|}{\lambda^*}
\le
\frac{C_{1,d}\sqrt{n\binom{n-2}{d-2}p}}{\lambda^*}
\le
\frac{C_{1,d}\sqrt{p}}{\sqrt{\log n}}\,
\frac{\sqrt{n^{d-1}\log n}}{\lambda^*}
=o(1),
\]
where we used \(\binom{n-2}{d-2}\le n^{d-2}\) and Lemma~\ref{lem:lamstar_scale}.
\end{proof}

We next control \(W^{(m)}u^{(-m)}\). For each \(m\in[n]\), define the vector Bernstein radius

\begin{definition}[Vector Bernstein radius for \(u^{(-m)}\)]
\label{def:vector_bernstein_radius}
Let
\[
C_{\mathrm{vec}}\coloneqq \frac{26}{3}(d-1),
\qquad
L_n\coloneqq \sqrt{p(1-p)\binom{n-2}{d-2}\log n}+\log n.
\]
Define
\[
\cL_n\coloneqq C_{\mathrm{vec}}\,L_n.
\]
\end{definition}

\begin{claim}[Normalized vector Bernstein radius]
\label{clm:vector_radius_normalized}
Assume
\eqref{eq:recovery_thm_scaling}. Then
\begin{equation}\label{eq:vector_radius_normalized}
\frac{\cL_n}{\lambda^*}
\le
o(1)\|u^*\|_\infty.
\end{equation}
\end{claim}

\begin{proof}
Since \(\|u^*\|_\infty^{-1}=\sqrt{k}\), Definition~\ref{def:vector_bernstein_radius} gives
\[
\frac{1}{\lambda^*\,\|u^*\|_\infty}\cL_n
=
\frac{C_{\mathrm{vec}}\sqrt{k}\,L_n}{\lambda^*}.
\]
By Definition~\ref{def:vector_bernstein_radius}
and \(\binom{n-2}{d-2}\le n^{d-2}\), we obtain
\[
\frac{1}{\lambda^*\,\|u^*\|_\infty}\cL_n
\le
C_{\mathrm{vec}}
\left(
\frac{\sqrt{p\,k\,n^{d-2}\log n}}{\lambda^*}
+
\frac{\sqrt{k} \log n}{\lambda^*}
\right).
\]
Lemma~\ref{lem:lamstar_scale} now gives
\[
\frac{1}{\lambda^*\,\|u^*\|_\infty}\cL_n
=o(1),
\]
which is equivalent to \eqref{eq:vector_radius_normalized}.
\end{proof}

We proceed to show that the following event has high probability:
\begin{equation}
E_{\mathrm{vec}}\coloneqq
\left\{
\max_{m\in[n]}\|W^{(m)}u^{(-m)}\|_2
\le
\cL_n
\right\}.
\label{eq:evec_event}
\end{equation}

\begin{claim}[Uniform bound for \(W^{(m)}u^{(-m)}\)]
\label{clm:Wm_um_uniform_pub}
The event \(E_{\mathrm{vec}}\) satisfies
\[
\P_S(E_{\mathrm{vec}})\ge 1-n^{-10}.
\]
\end{claim}

\begin{proof}
Fix \(m\in[n]\). Since \(u^{(-m)}=s_m\bar u^{(-m)}\) for some \(s_m\in\{\pm1\}\), we have
\[
\|W^{(m)}u^{(-m)}\|_2=\|W^{(m)}\bar u^{(-m)}\|_2.
\]
It therefore suffices to bound \(\|W^{(m)}\bar u^{(-m)}\|_2\).

For each \(e\in\cE_{m,\mathrm{out}}\), define
\[
Z_e\coloneqq (H_e-p)B_e\bar u^{(-m)}\in\R^n,
\qquad\text{so that}\qquad
W^{(m)}\bar u^{(-m)}=\sum_{e\in\cE_{m,\mathrm{out}}} Z_e.
\]
Since \(\bar u^{(-m)}\) is \(\cF^{(-m)}\)-measurable and, conditional on \(\cF^{(-m)}\), the variables \(\{H_e:e\in\cE_{m,\mathrm{out}}\}\) are independent \(\Ber(p)\), the vectors
\(\{Z_e\}_{e\in\cE_{m,\mathrm{out}}}\) are conditionally independent. Also,
\[
\E_S[Z_e\mid \cF^{(-m)}]
=
B_e\bar u^{(-m)}\,\E_S[H_e-p\mid \cF^{(-m)}]
=\0.
\]

We apply Lemma~\ref{lem:tropp-bernstein} conditionally on \(\cF^{(-m)}\) with \(d_1=n\) and \(d_2=1\). Set
\[
\tK\coloneqq d-1,
\qquad
\tV_n\coloneqq p(1-p)(d-1)^2\binom{n-2}{d-2},
\qquad
\ell\coloneqq 13\log n.
\]
By Lemma~\ref{lem:single_edge_action_bounds_pub}, we have
\begin{equation}\label{eq:vec_envelope_short_new}
\|Z_e\|_2
=
|H_e-p|\,\|B_e\bar u^{(-m)}\|_2
\le
|H_e-p|(d-1)\left(\sum_{j\in e}(\bar u_j^{(-m)})^2\right)^{1/2}
\le
\tK,
\end{equation}
where the last inequality uses \(|H_e-p|\le 1\) and \(\|\bar u^{(-m)}\|_2=1\).

Let \(\sigma_m^2\) denote the variance parameter in Lemma~\ref{lem:tropp-bernstein} for the sum \(\sum_{e\in\cE_{m,\mathrm{out}}}Z_e\). Since each \(Z_e\) is an \(n\times1\) matrix,
\[
\sigma_m^2
=
\max\left\{
\left\|\sum_{e\in\cE_{m,\mathrm{out}}}\E_S[Z_e^\top Z_e\mid \cF^{(-m)}]\right\|,
\left\|\sum_{e\in\cE_{m,\mathrm{out}}}\E_S[Z_eZ_e^\top\mid \cF^{(-m)}]\right\|
\right\}.
\]
Now \(Z_e^\top Z_e=\|Z_e\|_2^2\) is scalar, and \(Z_eZ_e^\top\succeq0\). Hence
\[
\left\|\sum_{e\in\cE_{m,\mathrm{out}}}\E_S[Z_e^\top Z_e\mid \cF^{(-m)}]\right\|
=
\sum_{e\in\cE_{m,\mathrm{out}}}\E_S[\|Z_e\|_2^2\mid \cF^{(-m)}],
\]
while, using \(\|A\|\le \tr(A)\) for \(A\succeq0\) together with \(\tr(zz^\top)=\|z\|_2^2\),
\[
\left\|\sum_{e\in\cE_{m,\mathrm{out}}}\E_S[Z_eZ_e^\top\mid \cF^{(-m)}]\right\|
\le
\tr\left(\sum_{e\in\cE_{m,\mathrm{out}}}\E_S[Z_eZ_e^\top\mid \cF^{(-m)}]\right)
=
\sum_{e\in\cE_{m,\mathrm{out}}}\E_S[\|Z_e\|_2^2\mid \cF^{(-m)}].
\]
Therefore
\begin{equation}
\sigma_m^2
\le
\sum_{e\in\cE_{m,\mathrm{out}}}\E_S[\|Z_e\|_2^2\mid \cF^{(-m)}].
\label{eq:init_sigma_bound}
\end{equation}
Using \eqref{eq:init_sigma_bound}, \eqref{eq:vec_envelope_short_new}, \(\E_S[(H_e-p)^2\mid \cF^{(-m)}]=p(1-p)\), and
\(\cE_{m,\mathrm{out}}\subseteq \cE_m\), we get
\begin{equation}
\label{eq:sum_um}
\sigma_m^2
\le
p(1-p)(d-1)^2
\sum_{e\in\cE_{m,\mathrm{out}}}\sum_{j\in e}(\bar u_j^{(-m)})^2
\le
p(1-p)(d-1)^2
\sum_{e\in\cE_m}\sum_{j\in e}(\bar u_j^{(-m)})^2.
\end{equation}
Since \(\bar u_m^{(-m)}=0\) and every \(e\in\cE_m\) contains \(m\),
\[
\sum_{e\in\cE_m}\sum_{j\in e}(\bar u_j^{(-m)})^2
=
\sum_{e\in\cE_m}\sum_{j\in e\setminus\{m\}}(\bar u_j^{(-m)})^2
=
\sum_{j\neq m}(\bar u_j^{(-m)})^2\,
\left|\{e\in\cE_m:\ j\in e\}\right|.
\]
For each fixed \(j\neq m\), the number of \(d\)-sets containing both \(m\) and \(j\) is
\(\binom{n-2}{d-2}\). Since \(\|\bar u^{(-m)}\|_2=1\),
\begin{equation}
\sum_{e\in\cE_m}\sum_{j\in e}(\bar u_j^{(-m)})^2
=
\binom{n-2}{d-2}\sum_{j\neq m}(\bar u_j^{(-m)})^2
\le
\binom{n-2}{d-2}.
\label{eq:sum_um_bound}
\end{equation}
Substituting \eqref{eq:sum_um_bound} into \eqref{eq:sum_um} gives
\begin{equation}
\label{eq:vec_varproxy_short_revised}
\sigma_m^2
\le
\tV_n.
\end{equation}

Lemma~\ref{lem:tropp-bernstein} therefore gives, for every \(t\ge 0\),
\[
\P_S\left(\|W^{(m)}\bar u^{(-m)}\|_2\ge t\,\middle|\,\cF^{(-m)}\right)
\le
(n+1)\exp\left(
-\frac{t^2}{2\tV_n+\frac23\tK t}
\right).
\]
Choose
\[
t_m\coloneqq \sqrt{2\tV_n\ell}+\frac{2}{3}\tK\,\ell.
\]
By Lemma~\ref{lem:bernstein_parameter_ineq}, the exponent is at least \(\ell\), and hence
\[
\P_S\left(\|W^{(m)}\bar u^{(-m)}\|_2\ge t_m\,\middle|\,\cF^{(-m)}\right)
\le
(n+1)e^{-\ell}.
\]
Taking expectations and a union bound over \(m\in[n]\) yields
\[
\P_S\left(\exists m\in[n]:\ \|W^{(m)}u^{(-m)}\|_2\ge t_m\right)
\le
n(n+1)e^{-\ell}
\le
n^{-10}.
\]

Finally,
\[
t_m
=
(d-1)\left(
\sqrt{26\,p(1-p)\binom{n-2}{d-2}\log n}
+\frac{26}{3}\log n
\right)
\]
\[
\le
\frac{26}{3}(d-1)\left(
\sqrt{p(1-p)\binom{n-2}{d-2}\log n}
+\log n
\right)
=
\cL_n,
\]
where we used \(\sqrt{26}\le 26/3\). Therefore,
\[
\P_S(E_{\mathrm{vec}})\ge 1-n^{-10},
\]
as claimed.
\end{proof}
\begin{lemma}[Deterministic bound for $P^{(m)}$]\label{lem:Pm_crude_pub}
Assume $d\ge3$ and $k\ge d$. If $m\notin S$, then $P^{(m)}=\0$. If $m\in S$, then
\begin{equation*}
\|P^{(m)}\|
\wle
2d\lambda^*\|u^*\|_\infty.
\end{equation*}
\end{lemma}

\begin{proof}
If $m\notin S$, then $\cE_{m,\mathrm{in}}=\emptyset$, hence $P^{(m)}=\0$.

Suppose that $m\in S$, and write $S_0\coloneqq S\setminus\{m\}$. Only planted hyperedges $e\subset S$ contribute to $P^{(m)}$, so $P^{(m)}$ is supported on $S\times S$ and therefore $\|P^{(m)}\|=\|P^{(m)}_{S,S}\|$.

Set
\begin{equation}
a_k\coloneqq (1-p)\binom{k-2}{d-2},
\qquad
b_k\coloneqq (1-p)\binom{k-3}{d-3}.
\label{eq:ak_bk_def}
\end{equation}
Then every entry of $P^{(m)}_{S,S}$ connecting $m$ to $S_0$ is equal to $a_k$, while the $S_0\times S_0$ block has diagonal entries zero and off-diagonal entries $b_k$. Thus, after reordering indices so that $m$ comes first,
\[
P^{(m)}_{S,S}
=
\begin{bmatrix}
0 & a_k \1_{k-1}^\top\\
a_k \1_{k-1} & b_k(J_{k-1}-I_{k-1})
\end{bmatrix} 
=
\underbrace{
\begin{bmatrix}
0 & a_k \1_{k-1}^\top\\
a_k \1_{k-1} & \0
\end{bmatrix}}_{\eqqcolon U^{(m)}}
+
\underbrace{
\begin{bmatrix}
0 & 0\\
0 & b_k(J_{k-1}-I_{k-1})
\end{bmatrix}}_{\eqqcolon V^{(m)}},
\]
where $\1_{k-1}\in\R^{k-1}$ is the all--ones vector and $J_{k-1},I_{k-1}\in\R^{(k-1)\times(k-1)}$ are the all--ones
and identity matrices. Hence $\|P^{(m)}\|\le \|U^{(m)}\|+\|V^{(m)}\|$.

The matrix $U^{(m)}$ has two nonzero singular values, both equal to $a_k\|\1_{k-1}\|_2=a_k\sqrt{k-1}$, so $\|U^{(m)}\|=a_k\sqrt{k-1}$. Also, $J_{k-1}-I_{k-1}$ is symmetric with eigenvalues $k-2$ and $-1$, hence $\|V^{(m)}\|=b_k(k-2)$. Using $(k-2)\binom{k-3}{d-3}=(d-2)\binom{k-2}{d-2}$ together with \eqref{eq:ak_bk_def}, we obtain
\[
\|V^{(m)}\|=b_k(k-2)=(d-2)a_k.
\]
Consequently,
\[
\|P^{(m)}\|
\le
a_k\sqrt{k-1}+(d-2)a_k
=
a_k\left(\sqrt{k-1}+d-2\right)
\le
d\,a_k\sqrt{k}.
\]
Since $\lambda^*=(k-1)a_k$, $\|u^*\|_\infty=1/\sqrt{k}$, and $\sqrt{k}/(k-1)\le 2/\sqrt{k}$ for $k\ge2$, it follows that
\begin{equation*}
\|P^{(m)}\|
\le
d\,\lambda^*\frac{\sqrt{k}}{k-1}
\le
2d\,\lambda^*\|u^*\|_\infty,
\qedhere
\end{equation*}
\end{proof}

\begin{lemma}[Uniform leave--one--out separation]\label{lem:delta_proportional_theorem_regime}
Fix $d\ge 3$ and $0<p<1$. Assume \eqref{eq:recovery_thm_scaling}. Then, on
$E_1\cap E_{\mathrm{op}}$,
\[
\min_{m\in[n]}\delta_m(\lambda)\wge (1+o(1))\lambda^*.
\]
\end{lemma}
\begin{proof}
Recall \(\delta_m(\lambda)=\min_{i\ge2}|\lambda_i(M^{(-m)})-\lambda|\), \(\lambda=\lambda_1(M)\), and \(\lambda^*=\lambda_1(M^*)\). The argument proving \eqref{eq:delta_m_def} yields, uniformly in \(m\),
\begin{equation}
\frac{\delta_m(\lambda)}{\lambda^*}
\wge
1-2\frac{\|M-M^*\|}{\lambda^*}-\frac{\|B^{(m)}\|}{\lambda^*}.
\label{eq:delta_bound_init}
\end{equation}
On \(E_1\), Lemma~\ref{lem:Delta0_vanish} yields \(\|M-M^*\|/\lambda^*=o(1)\).
Moreover, by Observation~\ref{obs:Bm_inout_decomp_pub} we have \(B^{(m)}=P^{(m)}+W^{(m)}\), so
\begin{equation}
\label{eq:bm_split}
\max_{m\in[n]}\frac{\|B^{(m)}\|}{\lambda^*}
\wle
\max_{m\in[n]}\frac{\|P^{(m)}\|}{\lambda^*}
+
\max_{m\in[n]}\frac{\|W^{(m)}\|}{\lambda^*}.
\end{equation}
By Lemma~\ref{lem:Pm_crude_pub},
\[
\max_m\frac{\|P^{(m)}\|}{\lambda^*}=o(1),
\]
and by Corollary~\ref{cor:Wm_op_littleo},
\[
\max_{m\in[n]}\frac{\|W^{(m)}\|}{\lambda^*}=o(1).
\]
Substituting these bounds into \eqref{eq:bm_split} and then into \eqref{eq:delta_bound_init} concludes the proof.
\end{proof}

Recall the definition of $\eta_m$ from the statement of Lemma \ref{lem:fixed_point_simple_rigorous}
\[
\eta_m= \frac{\sqrt2\,\|B^{(m)}\|}{\delta_m(\lambda)}.
\]

\begin{lemma}[Uniform bounds for \(B^{(m)}\)]
\label{lem:Bm_absorb_transfer_short}
Fix \(d\ge 3\) and \(0<p<1\). Assume \eqref{eq:recovery_thm_scaling}. On \(E_1\cap E_{\mathrm{op}}\cap E_{\mathrm{vec}}\),
\begin{enumerate}[label=(\alph*)]
\item\label{eq:Bm_op_absorb_short}
\(
\max_{m\in[n]}\|B^{(m)}\|\le o(1)\lambda^*.
\)
\item\label{eq:eta_half_short}
\(
\max_{m\in[n]}\eta_m=o(1).
\)
\item\label{eq:Bm_um_target_short}
\(
\max_{m\in[n]}\|B^{(m)}u^{(-m)}\|_2
\le
(2d+o(1))\lambda^*\,\|u^*\|_\infty.
\)
\item\label{eq:Bm_u_transfer_cor}
\(
\max_{m\in[n]}\|B^{(m)}u\|_2
\le
(2d+o(1))\lambda^*\,\|u^*\|_\infty.
\)
\end{enumerate}
\end{lemma}

\begin{proof}
Work on \(E_1\cap E_{\mathrm{op}}\cap E_{\mathrm{vec}}\). By Lemma~\ref{lem:delta_proportional_theorem_regime},
\[
\min_{m\in[n]}\delta_m(\lambda)\ge (1+o(1))\lambda^*.
\]
Moreover, since \(B^{(m)}=P^{(m)}+W^{(m)}\),
\[
\max_{m\in[n]}\frac{\|B^{(m)}\|}{\lambda^*}
\le
\max_{m\in[n]}\frac{\|P^{(m)}\|}{\lambda^*}
+
\max_{m\in[n]}\frac{\|W^{(m)}\|}{\lambda^*}
=
o(1),
\]
by Lemma~\ref{lem:Pm_crude_pub} and Corollary~\ref{cor:Wm_op_littleo}.
This proves \ref{eq:Bm_op_absorb_short}, and \ref{eq:eta_half_short} follows immediately from the definition of \(\eta_m\).

Also,
\[
\frac{\|M^{(-m)}-M^*\|}{\lambda^*}
\le
\frac{\|M-M^*\|}{\lambda^*}+\frac{\|B^{(m)}\|}{\lambda^*}
=
o(1),
\]
uniformly in \(m\). Since \(\lambda_1(M^*)=\lambda^*\) and \(\lambda_2(M^*)=0\), Weyl's inequality yields
\[
\lambda_1(M^{(-m)})\ge (1-o(1))\lambda^*,
\qquad
\lambda_2(M^{(-m)})\le o(1)\lambda^*,
\]
so \(\lambda_1(M^{(-m)})\) is simple for all sufficiently large \(n\). Thus Lemma~\ref{lem:fixed_point_simple_rigorous} applies.

Next, by Definition~\ref{eq:evec_event},
\[
\max_{m\in[n]}\|W^{(m)}u^{(-m)}\|_2
\le
\cL_n.
\]
Hence, using \(B^{(m)}=P^{(m)}+W^{(m)}\),
\[
\max_{m\in[n]}\|B^{(m)}u^{(-m)}\|_2
\le
\max_{m\in[n]}\|P^{(m)}\|+\cL_n.
\]
By Lemma~\ref{lem:Pm_crude_pub},
\[
\max_{m\in[n]}\|P^{(m)}\|
\le
2d\,\lambda^*\,\|u^*\|_\infty,
\]
and by Claim~\ref{clm:vector_radius_normalized},
\[
\cL_n
\le
o(1)\lambda^*\,\|u^*\|_\infty.
\]
This proves \ref{eq:Bm_um_target_short}.

Finally, Lemma~\ref{lem:fixed_point_simple_rigorous} together with \ref{eq:eta_half_short} yields
\[
\max_{m\in[n]}\|B^{(m)}u\|_2
\le
\frac{1}{1-\max_{m\in[n]}\eta_m}\,
\max_{m\in[n]}\|B^{(m)}u^{(-m)}\|_2
=
(1+o(1))\max_{m\in[n]}\|B^{(m)}u^{(-m)}\|_2,
\]
and \ref{eq:Bm_u_transfer_cor} follows from \ref{eq:Bm_um_target_short}.
\end{proof}

\subsubsection{Leave--one--out eigenvector comparison}
\label{sec:loo_ev_bound}

\begin{lemma}[Uniform leave--one--out sign alignment]
\label{lem:loo_sign_alignment_uniform}
Assume \eqref{eq:recovery_thm_scaling}. Then, for all sufficiently large $n$,
on $E_1\cap E_{\mathrm{op}}\cap E_{\mathrm{vec}}$, simultaneously for all $m\in[n]$,
\[
\langle u^{(-m)},u^*\rangle \ge 0.
\]
\end{lemma}

\begin{proof}
Work on $E_1\cap E_{\mathrm{op}}\cap E_{\mathrm{vec}}$ and write
\[
\Delta_0 = \frac{\|M-M^*\|}{\lambda^*},
\qquad
\Delta_m = \frac{\|M^{(-m)}-M^*\|}{\lambda^*}.
\]

Recall \(B^{(m)} = M-M^{(-m)}\).
Then, for every \(m\in[n]\),
\[
M^{(-m)}-M^*
\weq
(M-B^{(m)})-M^*
\weq
(M-M^*)-B^{(m)},
\]
and hence
\(
\|M^{(-m)}-M^*\|
\le
\|M-M^*\|+\|B^{(m)}\|,
\)
Then, for every $m\in[n]$,
\[
\Delta_m \le \Delta_0+\frac{\|B^{(m)}\|}{\lambda^*}.
\]
On $E_1$, Lemma~\ref{lem:Delta0_vanish} yields $\Delta_0=o(1)$, and
Lemma~\ref{lem:Bm_absorb_transfer_short}\ref{eq:Bm_op_absorb_short} yields
\[
\max_{m\in[n]}\frac{\|B^{(m)}\|}{\lambda^*}=o(1).
\]
Therefore,
\[
\max_{m\in[n]}(\Delta_0+\Delta_m)
\le
2\Delta_0+\max_{m\in[n]}\frac{\|B^{(m)}\|}{\lambda^*}
=o(1).
\]
In particular, for all sufficiently large $n$,
\[
\max_{m\in[n]} 2^{3/2}\,(\Delta_0+\Delta_m) < \sqrt2.
\]

Fix \(m\in[n]\). Since \(\langle u,u^*\rangle\ge 0\) by Remark~\ref{rem:sign},
Lemma~\ref{lem:DK_sign_aligned} gives
\[
\|u-u^*\|_2 \le 2^{3/2}\Delta_0.
\]
Applying Lemma~\ref{lem:DK_sign_aligned} to \((M^{(-m)},M^*)\), there exists \(s_m\in\{\pm1\}\) such that
\[
\langle s_m u^{(-m)},u^*\rangle\ge 0
\qquad\text{and}\qquad
\|s_m u^{(-m)}-u^*\|_2 \le 2^{3/2}\Delta_m.
\]
Hence
\begin{equation}
\|u-s_m u^{(-m)}\|_2
\le
\|u-u^*\|_2+\|s_m u^{(-m)}-u^*\|_2
\le
2^{3/2}(\Delta_0+\Delta_m)
<\sqrt2.
\label{eq:sign_alignment_cause}
\end{equation}

For unit vectors \(x,y\), since
\[
\|x-y\|_2^2=2-2\langle x,y\rangle,
\]
the inequality \(\|x-y\|_2<\sqrt2\) implies \(\langle x,y\rangle>0\). Applying this with
\(x=u\) and \(y=s_m u^{(-m)}\), \eqref{eq:sign_alignment_cause} gives
\[
\langle u,s_m u^{(-m)}\rangle>0.
\]
If \(s_m=-1\), then \(\langle u,u^{(-m)}\rangle<0\), contradicting
\(\langle u^{(-m)},u\rangle\ge 0\) from Remark~\ref{rem:sign}. Hence \(s_m=+1\), and therefore
\[
\langle u^{(-m)},u^*\rangle
=
\langle s_m u^{(-m)},u^*\rangle
\ge 0.
\]
Since the preceding bounds are uniform in \(m\), the claim follows for all \(m\in[n]\).
\end{proof}

\begin{lemma}[Leave--one--out eigenvector comparison]
\label{lem:loo-evec-comparison-simplified}
On the event \(E_1\cap E_{\mathrm{op}}\cap E_{\mathrm{vec}}\), and with \(k\) chosen according to \eqref{eq:recovery_thm_scaling},
\begin{enumerate}[label=(\alph*)]
\item \(\max_{m\in[n]}\|u^{(-m)}-u^*\|_2 \le o(1)\), \label{eq:loo-L2-simplified}
\item \(\max_{m\in[n]}\|u-u^{(-m)}\|_2
\le
(2\sqrt2d+o(1))\|u^*\|_\infty,\)
\label{eq:full-vs-loo-L2-simplified}
\item \(\max_{m\in[n]}\|u^{(-m)}-u^*\|_\infty \le 2.\)
\label{eq:loo-Linf-simplified}
\end{enumerate}
\end{lemma}

\begin{proof}
Fix \(m\in[n]\) and work on \(E_1\cap E_{\mathrm{op}}\cap E_{\mathrm{vec}}\).

By Lemma~\ref{lem:loo_sign_alignment_uniform}, for all sufficiently large \(n\) we have
\(\langle u^{(-m)},u^*\rangle\ge 0\), so Lemma~\ref{lem:DK_sign_aligned} yields
\[
\|u^{(-m)}-u^*\|_2 \le 2^{3/2}\frac{\|M^{(-m)}-M^*\|}{\lambda^*}.
\]
Using \(M^{(-m)}=M-B^{(m)}\) and the triangle inequality,
\[
\frac{\|M^{(-m)}-M^*\|}{\lambda^*}
\le
\frac{\|M-M^*\|}{\lambda^*}+\frac{\|B^{(m)}\|}{\lambda^*}
=
\Delta_0+\frac{\|B^{(m)}\|}{\lambda^*}.
\]
On \(E_1\), Lemma~\ref{lem:Delta0_vanish} gives \(\Delta_0=o(1)\), and on
\(E_1\cap E_{\mathrm{op}}\cap E_{\mathrm{vec}}\), Lemma~\ref{lem:Bm_absorb_transfer_short}\ref{eq:Bm_op_absorb_short}
gives \(\max_{m\in[n]}\|B^{(m)}\|/\lambda^*=o(1)\). Hence \ref{eq:loo-L2-simplified} holds.

By the proof of Lemma~\ref{lem:Bm_absorb_transfer_short}, \(\lambda_1(M^{(-m)})\) is simple on \(E_1\cap E_{\mathrm{op}}\cap E_{\mathrm{vec}}\). Also, Lemma~\ref{lem:delta_proportional_theorem_regime} gives \(\delta_m(\lambda)\ge (1+o(1))\lambda^*>0\). Therefore Lemma~\ref{lem:loo-evec-comparison-DK} applies and gives
\[
\|u-u^{(-m)}\|_2 \le \frac{\sqrt2}{\delta_m(\lambda)}\|B^{(m)}u\|_2.
\]
Using Lemma~\ref{lem:Bm_absorb_transfer_short}\ref{eq:Bm_u_transfer_cor}, we obtain
\[
\|u-u^{(-m)}\|_2
\le
\frac{\sqrt2}{(1+o(1))\lambda^*}(2d+o(1))\lambda^*\|u^*\|_\infty
=
(2\sqrt2d+o(1))\|u^*\|_\infty,
\]
proving \ref{eq:full-vs-loo-L2-simplified}.

Finally,
\[
\|u^{(-m)}-u^*\|_\infty
\le \|u^{(-m)}\|_\infty+\|u^*\|_\infty
\le \|u^{(-m)}\|_2+\|u^*\|_2
=2,
\]
which proves \ref{eq:loo-Linf-simplified}.
\end{proof}

\subsubsection{Row-wise concentration}
\label{sec:row-wise}

In this section, we provide a bound for the row concentration term 
\[
\max_{m\in[n]} \left|M_{m:}(u^{(-m)}-u^*)\right|
\]
in \eqref{eq:key_quantities}. We begin by introducing a few definitions and notations.
\begin{definition}
\label{def:scalar_bernstein_notation}
Denote
\[
V_n \coloneqq \sqrt{p(1-p)(d-1)\binom{n-2}{d-2}\log n},
\qquad
K_n\coloneqq (d-1)\log n.
\]
For \(v\in\R^n\), define the row Bernstein radius
\[
\mathcal R_n(v)\coloneqq V_n\|v\|_2+K_n\|v\|_\infty.
\]
\end{definition}

\begin{definition}[Event \(E_2\)]\label{def:event_E2}
Let \(\{v^{(-i)}\}_{i=1}^n\subset\R^n\). Define \(E_2(\{v^{(-i)}\}_{i=1}^n)\) as the event that
\begin{equation}\label{eq:E2_def}
\left|(M-M^*)_{i:}v^{(-i)}\right|
\;\le\;
8\,\mathcal R_n(v^{(-i)})
\qquad\text{for all }i\in[n].
\end{equation}
\end{definition}

\begin{remark}
If \(v^{(-i)} = u^*\) for all \(i\), we will simply denote the event by \(E_2(u^*)\).
\end{remark}

\begin{claim}[Row-Bernstein for \(u^*\)]\label{clm:row_bernstein_u_star}
\begin{equation}\label{eq:E2_prob_u_star}
\P_S\left(E_2(u^*)\right)\ge 1-n^{-10}.
\end{equation}
\end{claim}

\begin{proof}
For each \(i\in[n]\), define
\[
A_i
\coloneqq
\left\{
\left|(M-M^*)_{i:}u^*\right|
\le
8\,\mathcal R_n(u^*)
\right\}.
\]
Since \(u^*\) is deterministic under \(\P_S\), Lemma~\ref{lem:row-bernstein-general} applies directly. Now \(C=8\) satisfies
\[
8\ge \max\left\{\sqrt{2(11+1)},\frac23(11+1)\right\}
=\max\{\sqrt{24},8\},
\]
so Lemma~\ref{lem:row-bernstein-general}, Lemma~\ref{lem:row-bernstein-general}, applied with \(c=11\) and \(C=8\), gives
\[
\P_S(A_i^c)\le 2n^{-12}
\qquad\text{for each }i\in[n].
\]
Since
\[
E_2(u^*)=\bigcap_{i=1}^n A_i,
\]
a union bound gives, for all \(n\ge 2\),
\begin{equation*}
\P_S\left(E_2(u^*)^c\right)
\le
\sum_{i=1}^n \P_S(A_i^c)
\le
2n^{-11}
\le
n^{-10}. \qedhere
\end{equation*}
\end{proof}

\begin{proposition}[Normalized row Bernstein radius]\label{prop:row_radius_normalized}
Assume
\eqref{eq:recovery_thm_scaling}. Then, for every \(v\in\R^n\),
\begin{equation}\label{eq:Rn_over_lam_general}
\frac{\mathcal R_n(v)}{\lambda^*}
\wle
o(1)\|u^*\|_\infty\,\|v\|_2
+
o(1)\|u^*\|_\infty\,\|v\|_\infty,
\end{equation}
where the \(o(1)\) terms are deterministic and independent of \(v\). In particular,
\begin{equation}\label{eq:Rn_u_star_over_lam}
\frac{\mathcal R_n(u^*)}{\lambda^*}=o(\|u^*\|_\infty).
\end{equation}
\end{proposition}

\begin{proof}
By Definition~\ref{def:scalar_bernstein_notation},
\[
\mathcal R_n(v)=V_n\|v\|_2+K_n\|v\|_\infty,
\]
where
\[
V_n=\sqrt{p(1-p)(d-1)\binom{n-2}{d-2}\log n},
\qquad
K_n=(d-1)\log n.
\]
Since \(\|u^*\|_\infty^{-1}=\sqrt{k}\), we have
\[
\frac{V_n}{\lambda^*\,\|u^*\|_\infty}
=
\frac{\sqrt{k}\,V_n}{\lambda^*}
\le
\sqrt{d-1}\,
\frac{\sqrt{p\,k\,n^{d-2}\log n}}{\lambda^*}
=
o(1),
\]
where we used \(\binom{n-2}{d-2}\le n^{d-2}\) and Lemma~\ref{lem:lamstar_scale}. Also,
\[
\frac{K_n}{\lambda^*\,\|u^*\|_\infty}
=
(d-1)\frac{\sqrt{k}\,\log n}{\lambda^*}
=
o(1),
\]
again by Lemma~\ref{lem:lamstar_scale}.

Now for any \(v\in\R^n\),
\[
\frac{\mathcal R_n(v)}{\lambda^*}
=
\frac{V_n}{\lambda^*\,\|u^*\|_\infty}\,\|u^*\|_\infty\|v\|_2
+
\frac{K_n}{\lambda^*\,\|u^*\|_\infty}\,\|u^*\|_\infty\|v\|_\infty.
\]
Substituting the preceding two bounds yields \eqref{eq:Rn_over_lam_general}. Finally, taking
\(v=u^*\) and using \(\|u^*\|_2=1\) gives \eqref{eq:Rn_u_star_over_lam}.
\end{proof}

\begin{lemma}\label{lem:row-action-vstar_normalized_rates}
Fix \(d\ge 3\) and \(0<p<1\), and assume \eqref{eq:recovery_thm_scaling}. On the event \(E_2\left(u^*\right)\),
\[
\frac{\|Mu^*\|_\infty}{\lambda^*}\wle (1+o(1))\|u^*\|_\infty .
\]
\end{lemma}

\begin{proof}
Since \(M^*u^*=\lambda^*u^*\),
\[
\frac{\|Mu^*\|_\infty}{\lambda^*}
\wle
\frac{\|(M-M^*)u^*\|_\infty}{\lambda^*}+\|u^*\|_\infty
\wle
\|u^*\|_\infty+\frac{8}{\lambda^*}\mathcal R_n(u^*).
\]
Using Proposition~\ref{prop:row_radius_normalized} 
concludes the proof.
\end{proof}

\begin{lemma}[Two-to-infinity bound]\label{lem:2toinf-M}
Assume \eqref{eq:recovery_thm_scaling}. Then:
\begin{enumerate}[label=(\alph*)]
\item \(\|M^*\|_{2\to\infty}=(1+o(1))\lambda^*\,\|u^*\|_\infty\).
\label{eq:pop_2_to_inf}
\item On the event \(E_1\), \(\|M\|_{2\to\infty}=o(\lambda^*)\),
\label{eq:2toinf-M-new}
\end{enumerate}
\end{lemma}

\begin{proof}
Using \(M^*=(w_{\mathrm{in}}-w_{\mathrm{out}})(\1_S\1_S^\top-I_S)\), rows with \(i\notin S\) vanish, while for \(i\in S\) the row has \((k-1)\) nonzero entries equal to \(w_{\mathrm{in}}-w_{\mathrm{out}}\). Hence
\[
\|M^*\|_{2\to\infty}
=\sqrt{k-1}\,(w_{\mathrm{in}}-w_{\mathrm{out}})
=\frac{\lambda^*}{\sqrt{k-1}}
=(1+o(1))\lambda^*\,\|u^*\|_\infty,
\]
which proves \ref{eq:pop_2_to_inf}. Also, for any matrix \(A\), \(\|A\|_{2\to\infty}=\max_i\|A_{i:}\|_2\le \|A\|\), so on \(E_1\),
\[
\frac{\|M\|_{2\to\infty}}{\lambda^*}
\wle
\frac{\|M-M^*\|}{\lambda^*}+\frac{\|M^*\|_{2\to\infty}}{\lambda^*}.
\]
By Lemma~\ref{lem:Delta0_vanish}, the first term is \(o(1)\), and by \ref{eq:pop_2_to_inf} the second term is \(1/\sqrt{k-1}=o(1)\), since \(k\to\infty\) under \eqref{eq:recovery_thm_scaling}. This proves \ref{eq:2toinf-M-new}.
\end{proof}

\begin{lemma}[Normalized row transform]\label{lem:row_transform_normalized}
Assume \eqref{eq:recovery_thm_scaling}. Let \(\{v^{(-m)}\}_{m=1}^n\subset\R^n\). Then, on \(E_2\left(\{v^{(-m)}\}_{m=1}^n\right)\),
\[
\max_{m\in[n]}\frac{|M_{m:}v^{(-m)}|}{\lambda^*}
\wle
(1+o(1))\,\|u^*\|_\infty\max_{m\in[n]}\|v^{(-m)}\|_2
+
o(1)\|u^*\|_\infty\max_{m\in[n]}\|v^{(-m)}\|_\infty.
\]
\end{lemma}

\begin{proof}
Fix \(m\in[n]\). By the triangle inequality,
\[
|M_{m:}v^{(-m)}|
\le
|(M-M^*)_{m:}v^{(-m)}|+|M^*_{m:}v^{(-m)}|.
\]
On \(E_2\left(\{v^{(-m)}\}_{m=1}^n\right)\),
\[
|(M-M^*)_{m:}v^{(-m)}|\le 8\,\mathcal R_n(v^{(-m)}).
\]
Dividing by \(\lambda^*\) and using Proposition~\ref{prop:row_radius_normalized} gives
\[
\frac{|(M-M^*)_{m:}v^{(-m)}|}{\lambda^*}
\wle
o(1)\|u^*\|_\infty\,\|v^{(-m)}\|_2
+
o(1)\|u^*\|_\infty\,\|v^{(-m)}\|_\infty.
\]
Also, since \(M^*\) is symmetric, \(M^*_{m:}=(M^*_{:m})^\top\), so by the Cauchy--Schwarz inequality \eqref{ineq:cs},
\[
|M^*_{m:}v^{(-m)}|
=
|\langle M^*_{:m},v^{(-m)}\rangle|
\wle
\|M^*_{:m}\|_2\,\|v^{(-m)}\|_2
\wle
\|M^*\|_{2\to\infty}\,\|v^{(-m)}\|_2.
\]
Using Lemma~\ref{lem:2toinf-M}\ref{eq:pop_2_to_inf}, we obtain
\[
\frac{|M^*_{m:}v^{(-m)}|}{\lambda^*}
\wle
(1+o(1))\,\|u^*\|_\infty\,\|v^{(-m)}\|_2.
\]
Combining the last two displays and taking the maximum over \(m\) proves the claim.
\end{proof}

\begin{claim}
\label{clm:PS_E2_u_m}
\[
\P_S\left(E_2\left(\{u^{(-m)}-u^*\}_{m=1}^n\right)\right)\ge 1-2n^{-10}.
\]
\end{claim}
\begin{proof}
For each \(m\in[n]\), define
\[
A_m^{\pm}
\coloneqq
\left\{
\left|(M-M^*)_{m:}(\pm \bar u^{(-m)}-u^*)\right|
\le
8\,\mathcal R_n(\pm \bar u^{(-m)}-u^*)
\right\}.
\]
Since \(\bar u^{(-m)}\) is \(\cF^{(-m)}\)-measurable, both
\(\pm \bar u^{(-m)}-u^*\) are \(\cF^{(-m)}\)-measurable. Therefore
Lemma~\ref{lem:row-bernstein-general}, applied conditionally on \(\cF^{(-m)}\) with \(c=11\) and \(C=8\), gives
\[
\P_S\left((A_m^{\pm})^c\,\middle|\,\cF^{(-m)}\right)\le 2n^{-12}
\qquad\text{a.s.}
\]
Taking expectations,
\[
\P_S\left((A_m^{\pm})^c\right)\le 2n^{-12}.
\]

Since \(u^{(-m)}=s_m\bar u^{(-m)}\) for some \(s_m\in\{\pm1\}\), we have, for every realization,
\[
\left\{
\left|(M-M^*)_{m:}(u^{(-m)}-u^*)\right|
>
8\,\mathcal R_n(u^{(-m)}-u^*)
\right\}
\subseteq
(A_m^+)^c\cup (A_m^-)^c.
\]
Hence
\[
\P_S\left(
\left|(M-M^*)_{m:}(u^{(-m)}-u^*)\right|
>
8\,\mathcal R_n(u^{(-m)}-u^*)
\right)
\le 4n^{-12}.
\]
A union bound over \(m\in[n]\) yields
\[
\P_S\left(E_2\left(\{u^{(-m)}-u^*\}_{m=1}^n\right)^c\right)
\le 4n^{-11}\le 2n^{-10}
\]
for all \(n\ge 2\), proving the claim.
\end{proof}

\begin{lemma}[Uniform leave--one--out row transform]\label{lem:loo-row-action}
Assume \eqref{eq:recovery_thm_scaling}. On the event
\[
E_1\cap E_{\mathrm{op}}\cap E_{\mathrm{vec}}\cap E_2\left(\{u^{(-m)}-u^*\}_{m=1}^n\right),
\]
for all sufficiently large \(n\),
\[
\max_{m\in[n]}\frac{|M_{m:}(u^{(-m)}-u^*)|}{\lambda^*}
=o(\|u^*\|_\infty).
\]
\end{lemma}

\begin{proof}
Work on \(E_1\cap E_{\mathrm{op}}\cap E_{\mathrm{vec}}\cap E_2\left(\{u^{(-m)}-u^*\}_{m=1}^n\right)\). Applying Lemma~\ref{lem:row_transform_normalized} with \(v^{(-m)}=u^{(-m)}-u^*\) yields
\begin{align}
\max_{m\in[n]}\frac{|M_{m:}(u^{(-m)}-u^*)|}{\lambda^*}
\le~
(1+o(1))\,&\|u^*\|_\infty\max_{m\in[n]}\|u^{(-m)}-u^*\|_2\nonumber\\
+~
o(1)&\|u^*\|_\infty\max_{m\in[n]}\|u^{(-m)}-u^*\|_\infty.
\label{eq:loo-row-action}
\end{align}
By Lemma~\ref{lem:loo-evec-comparison-simplified}\ref{eq:loo-L2-simplified},
\begin{equation}
\max_{m\in[n]}\|u^{(-m)}-u^*\|_2=o(1),
\label{eq:nom_bound_l2}
\end{equation}
and by Lemma~\ref{lem:loo-evec-comparison-simplified}\ref{eq:loo-Linf-simplified},
\begin{equation}
\max_{m\in[n]}\|u^{(-m)}-u^*\|_\infty\le 2.
\label{eq:nom_bound_sup}
\end{equation}
Substituting \eqref{eq:nom_bound_l2} and \eqref{eq:nom_bound_sup} into \eqref{eq:loo-row-action} concludes the proof.
\end{proof}

\subsubsection{Conclusion of the proof}
In this section, we prove Theorem \ref{thm:recovery_sim} using the results from the previous sections.
\begin{lemma}\label{lem:one-shot}
Assume \eqref{eq:recovery_thm_scaling}. On the event
\[
E_1 \cap E_2\left(u^*\right)\cap E_2\left(\{u^{(-m)}-u^*\}_{m=1}^n\right)\cap E_{\mathrm{op}}\cap E_{\mathrm{vec}},
\]
for all sufficiently large \(n\),
\[
\left\|u-\frac{Mu^*}{\lambda^*}\right\|_\infty
= o(1)\|u^*\|_\infty.
\]
\end{lemma}

\begin{proof}
Since \(\Delta_0=o(1)\) on \(E_1\) by Lemma~\ref{lem:Delta0_vanish}, we have \(\Delta_0\le 1/2\) for all sufficiently large \(n\). Hence Proposition~\ref{prop:chain-Mu-star} applies. Let \(r_n\) be as in Proposition~\ref{prop:chain-Mu-star}. Under \(E_1\), Proposition~\ref{prop:chain-Mu-star} gives
\begin{equation}
\left\|u-\frac{Mu^*}{\lambda^*}\right\|_\infty
\le
2\Delta_0\,\frac{\|Mu^*\|_\infty}{\lambda^*}+2r_n.
\label{eq:prop:chain-Mu-star_rep}
\end{equation}
By Lemma~\ref{lem:loo-row-action},
\[
\frac{1}{\lambda^*}\max_{m\in[n]}|M_{m:}(u^{(-m)}-u^*)|
=o(\|u^*\|_\infty),
\]
and by Lemma~\ref{lem:2toinf-M}\ref{eq:2toinf-M-new} together with Lemma~\ref{lem:loo-evec-comparison-simplified}\ref{eq:full-vs-loo-L2-simplified},
\[
\frac{\|M\|_{2\to\infty}}{\lambda^*}\max_{m\in[n]}\|u-u^{(-m)}\|_2
=
o(1)\cdot O(\|u^*\|_\infty)
=
o(\|u^*\|_\infty).
\]
Hence
\begin{equation}
r_n=o(\|u^*\|_\infty).
\label{eq:r_bound}
\end{equation}
Since, by Lemma~\ref{lem:row-action-vstar_normalized_rates},
\begin{equation}
\frac{\|Mu^*\|_\infty}{\lambda^*}\wle (1+o(1))\|u^*\|_\infty,
\label{eq:Mu_bound}
\end{equation}
and \(\Delta_0=o(1)\) on \(E_1\), substituting \eqref{eq:r_bound} and \eqref{eq:Mu_bound} into \eqref{eq:prop:chain-Mu-star_rep} gives
\[
\left\|u-\frac{Mu^*}{\lambda^*}\right\|_\infty=o(\|u^*\|_\infty),
\]
which proves the claim.
\end{proof}

\begin{lemma}\label{lem:pivot-Mu-star}
Assume \eqref{eq:recovery_thm_scaling}. Then, on \(E_2\left(u^*\right)\), for all sufficiently large \(n\),
\[
\left\|\frac{Mu^*}{\lambda^*}-u^*\right\|_\infty
=o(1)\|u^*\|_\infty.
\]
\end{lemma}

\begin{proof}
Since \(M^*u^*=\lambda^*u^*\),
\[
\frac{Mu^*}{\lambda^*}-u^*
=
\frac{(M-M^*)u^*}{\lambda^*}.
\]
On \(E_2\left(u^*\right)\), by Definition~\ref{def:event_E2},
\[
\left\|\frac{Mu^*}{\lambda^*}-u^*\right\|_\infty
\wle
\frac{8}{\lambda^*}\mathcal R_n(u^*).
\]
Using Proposition~\ref{prop:row_radius_normalized} concludes the proof.
\end{proof}
\label{sec:recovery_conclusion}

Using Lemma \ref{lem:one-shot} and \ref{lem:pivot-Mu-star} we now prove Theorem \ref{thm:recovery_sim}.
\begin{proof}[Proof of Theorem \ref{thm:recovery_sim}]
    Define the good event
\[
G \coloneqq
E_1
\cap E_2\left(u^*\right)
\cap E_2\left(\{u^{(-m)}-u^*\}_{m=1}^n\right)
\cap E_{\mathrm{op}}
\cap E_{\mathrm{vec}}.
\]
By a union bound together with Claims~\ref{clm:PSE1},~\ref{clm:Wm_op_uniform_lkc},~\ref{clm:Wm_um_uniform_pub}, \ref{clm:row_bernstein_u_star},~\ref{clm:PS_E2_u_m}, we obtain
\begin{equation}
\P_S(G)\ge 1-10n^{-10}.
\label{eq:PSG_union_bound}
\end{equation}
By Lemma~\ref{lem:one-shot} and Lemma~\ref{lem:pivot-Mu-star}, there exists a deterministic sequence
\(\varepsilon_n = o(1)\) such that on \(G\),
\[
\left\|\frac{Mu^*}{\lambda^*}-u^*\right\|_\infty 
+ \left\|u-\frac{Mu^*}{\lambda^*}\right\|_\infty
\le \frac{\varepsilon_n}{\sqrt{k}}.
\]
Therefore
\[
\P_S\left(\left\|\frac{Mu^*}{\lambda^*}-u^*\right\|_\infty 
+ \left\|u-\frac{Mu^*}{\lambda^*}\right\|_\infty \le \frac{\varepsilon_n}{\sqrt{k}}\right)\to1.
\]
Lemma~\ref{lem:suff_cond1} now yields exact recovery and completes the proof of Theorem~\ref{thm:recovery_sim}.
\end{proof}

\begin{proof}[Proof of Corollary~\ref{cor:pn_extensions_recovery}]
Interpret \eqref{eq:recovery_thm_scaling} with \(p\) replaced by \(p_n\), i.e.
\[
k\gg
\left(\frac{p_n}{1-p_n}\right)^{\frac{1}{2(d-1)}}\sqrt n.
\]
Since \(\sup_n p_n<1\) and
\[
p_n=\Omega\left(n^{-(d-1)}\log^{c_d} n\right)
\]
with \(c_d\ge \dfrac{4(d-1)}{2d-3}\), we have \(k\to\infty\). The assumption \eqref{eq:alpha_cond} is satisfied, so Corollary~\ref{cor:lkc_mod} holds and the probability bounds for \(E_1\), \(E_2\), \(E_{\mathrm{op}}\), and \(E_{\mathrm{vec}}\) remain valid verbatim with \(p_n\) in place of \(p\).

Moreover, by the \(p_n\)-extension in Lemma~\ref{lem:lamstar_scale}, the scale relations in \eqref{eq:lamstar_scale} continue to hold with \(p_n\) in place of \(p\). Every subsequent estimate in the proof of Theorem~\ref{thm:recovery_sim} uses \(p\) only through those scale relations, and therefore goes through without further modification. Hence the conclusion of Theorem~\ref{thm:recovery_sim} remains valid with \(p=p_n\).
\end{proof}
\addcontentsline{toc}{section}{References}

\bibliography{refs_hidden_clique.bib}
\bibliographystyle{abbrv}

\newpage
\appendix

\section{Perturbation and concentration}
\label{sec:concentration_inequalities}

This appendix collects standard inequalities and auxiliary perturbation/concentration results used throughout the paper.

\subsection{Basic inequalities}

\paragraph{Cauchy--Schwarz inequality.}
For $x,y\in\R^n$,
\begin{equation}
|\langle x,y\rangle|
\;\le\;
\|x\|_2\,\|y\|_2 .
\label{ineq:cs}
\end{equation}

\paragraph{Weyl's inequality.}
For a symmetric matrix \(C\in\R^{n\times n}\), let
\[
\lambda_1(C)\ge \cdots \ge \lambda_n(C)
\]
denote its eigenvalues listed in nonincreasing order. Then, for any two symmetric matrices
\(A,B\in\R^{n\times n}\),
\begin{equation}
\max_{i\in[n]} \left|\lambda_i(A)-\lambda_i(B)\right|
\le
\|A-B\|.
\label{ineq:weyl}
\end{equation}
\subsection{Bernstein inequalities}

Next, we record scalar and matrix versions of Bernstein's inequality, together with a simple parameter inequality used repeatedly to instantiate Bernstein's bound.

\begin{lemma}[Scalar Bernstein inequality {\cite[Thm.~2.8.4]{vershynin2018high}}]
\label{lem:scalar-bernstein}
Let $X_1,\dots,X_N$ be independent real-valued random variables with $\E[X_i]=0$ and
$|X_i|\le K$ almost surely for all $i$. Define the variance proxy
\[
\sigma^2 \;\coloneqq\; \sum_{i=1}^N \E[X_i^2].
\]
Then, for every $t\ge 0$,
\[
\P\left(\left|\sum_{i=1}^N X_i\right|\ge t\right)
\wle
2\exp\left(
-\frac{t^2}{2\sigma^2+\frac{2}{3}Kt}
\right).
\]
\end{lemma}

\begin{lemma}[Matrix Bernstein, rectangular case {\cite[Theorem~1.6]{tropp2012user}}]
\label{lem:tropp-bernstein}
Let $\{Z_k\}$ be a finite collection of independent random matrices in
$\R^{d_1\times d_2}$ such that $\E[Z_k]=0$ for all $k$ and
\[
\|Z_k\|\le R \qquad \text{almost surely.}
\]
Define the (rectangular) variance parameter
\[
\sigma^2
\;\coloneqq\;
\max\left\{
\left\|\sum_k \E[Z_k Z_k^\top]\right\|,
\;\;
\left\|\sum_k \E[Z_k^\top Z_k]\right\|
\right\}.
\]
Then, for every $t\ge 0$,
\[
\P\left(\left\|\sum_k Z_k\right\|\ge t\right)
\wle
(d_1+d_2) 
\exp\left(
-\frac{t^2}{2\sigma^2+\frac{2}{3}Rt}
\right).
\]
\end{lemma}

\begin{lemma}[Bernstein parameter inequality]\label{lem:bernstein_parameter_ineq}
Fix $\ell>0$ and let $V\ge 0$ and $U\ge 0$. Define
\[
t \;\coloneqq\; \sqrt{2V\ell}+\frac{2}{3}U\ell.
\]
Then
\begin{equation}\label{eq:bernstein_parameter_ineq_general}
t^2 \;\ge\; \ell\left(2V+\frac{2}{3}Ut\right).
\end{equation}
\end{lemma}

\begin{proof}
Write $t=a+b$ with $a\coloneqq \sqrt{2V\ell}$ and $b\coloneqq \frac{2}{3}U\ell$. Then
\[
t^2\weq (a+b)^2
\weq a^2+2ab+b^2
\weq 2V\ell+\frac{4}{3}U\ell\,\sqrt{2V\ell}+\frac{4}{9}U^2\ell^2.
\]
On the other hand,
\[
\ell\left(2V+\frac{2}{3}Ut\right)
\weq 2V\ell+\frac{2}{3}U\ell(a+b)
\weq 2V\ell+\frac{2}{3}U\ell\,\sqrt{2V\ell}+\frac{4}{9}U^2\ell^2.
\]
Subtracting gives
\[
t^2-\ell\left(2V+\frac{2}{3}Ut\right)
=
\left(\frac{4}{3}-\frac{2}{3}\right)U\ell\,\sqrt{2V\ell}
=
\frac{2}{3}U\ell\,\sqrt{2V\ell}
\;\ge\;0,
\]
which proves \eqref{eq:bernstein_parameter_ineq_general}.
\end{proof}

\subsection{Concentration of the adjacency matrix}

For \(e\in\binom{[n]}{d}\), let \(B_e\in\R^{n\times n}\) consider the co-membership matrix
\[
(B_e)_{ij} = \1\{i\neq j,\ \{i,j\}\subset e\}.
\]
Fix \(d\ge 2\). Let \(\{X_e\}_{e\in\binom{[n]}{d}}\) be independent random variables satisfying \(0\le X_e\le 1\) and \(\E[X_e]\le p_n\) for all \(e\in\binom{[n]}{d}\). Define
\[
A(\mathcal I)\coloneqq \sum_{e\in\mathcal I} X_e\,B_e,
\]
for any deterministic index set \(\mathcal I\subseteq \binom{[n]}{d}\), and write \(\E\) for expectation under the joint law of \(\{X_e\}_{e\in\binom{[n]}{d}}\). \cite[Theorem~4]{lee_kim_chung_2020} controls
\[
\|A(\cE)-\E[A(\cE)]\|,
\qquad
\cE = \binom{[n]}{d},
\]
and its proof uses only independence, boundedness, the uniform mean bound
\(\E[X_e]\le p_n\), and the overlap count
\(|\{e\in\cE:\{i,j\}\subset e\}|=\binom{n-2}{d-2}\).
For later applications, we need the same conclusion for deterministic
subsets \(\mathcal I\subseteq\cE\). This is achieved by applying the theorem to the masked family
\(\tX_e\coloneqq X_e\,\1\{e\in\mathcal I\}\) on the full index set \(\cE\).

\begin{lemma}[Adapted from \cite{lee_kim_chung_2020}, Theorem~4]\label{lem:lkc_mod}
Fix \(d\ge 2\) and \(\alpha>0\), and assume
\begin{equation}\label{eq:alpha_cond}
p_n \;\ge\; \alpha\,\frac{\log n}{n\binom{n-2}{d-2}}.
\end{equation}
Then there exists a constant \(C_{\alpha,d}>0\), depending only on \((\alpha,d)\), such that for every
deterministic index set \(\mathcal I\subseteq \binom{[n]}{d}\),
\begin{equation}\label{eq:lkc_general_index}
\left\|A(\mathcal I)-\E[A(\mathcal I)]\right\|
\;\le\;
C_{\alpha,d}\,\sqrt{n\binom{n-2}{d-2}\,p_n}
\end{equation}
with probability at least \(1-4n^{-11}\).
\end{lemma}

\begin{proof}
Fix deterministic \(\mathcal I\subseteq \binom{[n]}{d}\) and consider
\[
\tX_e = X_e\,\1\{e\in\mathcal I\},
\qquad
\tA(\mathcal I) = \sum_{e\in\cE}\tX_e\,B_e .
\]
Since \(\1\{e\in\mathcal I\}\) is deterministic and \(\{X_e\}_{e\in\cE}\) are independent,
the family \(\{\tX_e\}_{e\in\cE}\) is independent. Moreover \(0\le \tX_e\le 1\) and
\[
\E[\tX_e]=\E[X_e]\1\{e\in\mathcal I\}\le p_n
\qquad(e\in\cE).
\]
Thus the boundedness and mean assumptions used in \cite[Theorem~4]{lee_kim_chung_2020}
remain valid under masking. For each \(i\neq j\),
\[
\E\left[\sum_{e\in\cE:\,\{i,j\}\subset e}\tX_e\right]
=
\sum_{e\in\cE:\,\{i,j\}\subset e}\E[X_e]\1\{e\in\mathcal I\}
\le
p_n |\{e\in\cE:\{i,j\}\subset e\}|
=
p_n\binom{n-2}{d-2},
\]
so the same pair-overlap bound used in the reference also holds. Finally,
\[
\tA(\mathcal I)=\sum_{e\in\mathcal I}X_eB_e=A(\mathcal I),
\qquad
\E[\tA(\mathcal I)]=\E[A(\mathcal I)].
\]
Hence the masked family \(\{\tX_e\}_{e\in\cE}\) satisfies the hypotheses used in
\cite[Theorem~4]{lee_kim_chung_2020} on the full index set \(\cE\), and applying that theorem yields
\eqref{eq:lkc_general_index}.
\end{proof}

\paragraph{Planted model and adjacency matrix.}
Fix \(S\subset[n]\) and work under \(\P_S(\cdot)=\P(\cdot\mid S)\).
Let \(H\) be a \(d\)-uniform hypergraph with \(H_e=1\) for \(e\subset S\) and independent \(H_e\sim\Ber(p_n)\) for \(e\not\subset S\).
Define the adjacency matrix \(A\in\R^{n\times n}\) as in \eqref{eq:def_adjacency_matrix}
and write \(\E_S\) for expectation under \(\P_S\). Then
\[
A=\sum_{e\in\binom{[n]}{d}} H_e B_e.
\]
Split hyperedges into
\[
\Ein = \{e:e\subset S\},
\qquad
\Eout = \binom{[n]}{d}\setminus \Ein,
\]
so
\[
\E_S[A]=\sum_{e\in\Ein} B_e+\sum_{e\in\Eout} p_n B_e,
\qquad
A-\E_S[A]=\sum_{e\in\Eout}(H_e-p_n)\,B_e.
\]

\begin{corollary}\label{cor:lkc_mod}
Assume \eqref{eq:alpha_cond}. With \(\P_S\)-probability at least \(1-4n^{-11}\),
\[
\|A-\E_S[A]\|
\;\le\;
C_{\alpha,d}\,\sqrt{n\binom{n-2}{d-2}\,p_n}.
\]
\end{corollary}

\begin{proof}
Under \(\P_S\), the set \(\Eout\) is deterministic, the family \(\{H_e\}_{e\in\Eout}\) is independent \(\Ber(p_n)\), and
\[
A-\E_S[A]=\sum_{e\in\Eout}(H_e-p_n)\,B_e
= A(\Eout)-\E_S[A(\Eout)].
\]
Applying Lemma~\ref{lem:lkc_mod} with \(\mathcal I=\Eout\) gives the claim.
\end{proof}


\subsection{Davis--Kahan-type perturbation bounds}
\begin{lemma}[Residual-form Davis--Kahan $\sin\theta$ bound; cf. {\cite[Theorem~3]{deng2021strong}}]
\label{lem:DK_residual_subspace}
Let $M\in\R^{n\times n}$ be symmetric and admit an orthogonal eigendecomposition
\[
M = X\Lambda X^\top,\qquad X=[X_1\ \ X_2],\qquad \Lambda=\diag(\Lambda_1,\Lambda_2).
\]
Fix $\hat\lambda\in\R$ and a unit vector $\hat u\in\R^n$, and define
\[
\delta(\hat\lambda)\;\coloneqq\;\min_{\mu\in\spec(\Lambda_2)}|\mu-\hat\lambda|,
\qquad\text{assume }\delta(\hat\lambda)>0,
\]
where $\spec(\Lambda_2)$ denotes the set of diagonal entries of $\Lambda_2$.
Let $\theta\in[0,\pi/2]$ be defined by
\(
\sin\theta \coloneqq \|X_2^\top \hat u\|_2.
\)
Then
\begin{equation}
\label{eq:DK_residual_subspace}
\sin\theta
\wle 
\frac{\|(M-\hat\lambda I)\hat u\|_2}{\delta(\hat\lambda)}.
\end{equation}
\end{lemma}

\begin{corollary}\label{cor:DK_residual_rankone}
Let $A\in\R^{n\times n}$ be symmetric with leading eigenvalue $\lambda_1(A)$ of multiplicity one and corresponding unit eigenvector $v$.
Fix $\hat\lambda\in\R$ and a unit vector $x\in\R^n$, and set
\[
\delta(\hat\lambda) = \min_{i\ge 2}|\lambda_i(A)-\hat\lambda|,
\qquad\text{assume }\delta(\hat\lambda)>0.
\]
Choose the sign of $v$ so that $\langle x,v\rangle\ge 0$. Then
\begin{equation}\label{eq:DK_residual_rankone}
\|x-v\|_2
\wle 
\frac{\sqrt2}{\delta(\hat\lambda)}\,\|(A-\hat\lambda I)x\|_2.
\end{equation}
\end{corollary}

\begin{proof}
Let $X_1\coloneqq v$ and choose $X_2\in\R^{n\times(n-1)}$ with orthonormal columns spanning $v^\perp$, the orthogonal complement of $v$, so that
$X\coloneqq [X_1\ \ X_2]$ is orthogonal and
\[
A=X\diag(\lambda_1(A),\Lambda_2)X^\top,
\qquad
\spec(\Lambda_2)=\{\lambda_i(A):i\ge2\}.
\]
Define $\theta\in[0,\pi/2]$ by $\sin\theta\coloneqq\|X_2^\top x\|_2$. Since $X_2X_2^\top=I-vv^\top$ and $\|x\|_2=1$,
\[
\sin^2\theta=\|X_2^\top x\|_2^2=x^\top X_2X_2^\top x=x^\top(I-vv^\top)x=1-\langle x,v\rangle^2.
\]
Since $\theta\in[0,\pi/2]$ and $\langle x,v\rangle\ge0$, we have $\cos\theta=\langle x,v\rangle \in [0,1]$, and hence
\begin{equation}
\|x-v\|_2^2
=2-2\langle x,v\rangle
=2-2\cos\theta
\le 2(1-\cos^2\theta)
=2\sin^2\theta,
\quad\text{so}\quad
\|x-v\|_2\le \sqrt2 \sin\theta.
\label{eq:l2_norm_sqrt2_sin}
\end{equation}
Applying Lemma~\ref{lem:DK_residual_subspace} with $(M,\hat u)=(A,x)$ yields
\begin{equation}
\sin\theta \wle \frac{\|(A-\hat\lambda I)x\|_2}{\delta(\hat\lambda)}.
\label{eq:sin_delta}
\end{equation}
Combining \eqref{eq:l2_norm_sqrt2_sin} and \eqref{eq:sin_delta} gives \eqref{eq:DK_residual_rankone}.
\end{proof}

\begin{lemma}[Davis--Kahan bound with sign alignment; cf. {\cite[Corollary~1]{yu2015useful}}]
\label{lem:DK_sign_aligned}
Let $M^*\in\R^{n\times n}$ be symmetric with eigenvalues
$\lambda_1(M^*)\ge \lambda_2(M^*)\ge \cdots \ge \lambda_n(M^*)$ and unit eigenvector $u^*$
associated with $\lambda_1(M^*)$. Assume the top eigenvalue has multiplicity one and define the eigengap
\[
\gamma^* \;\coloneqq\; \lambda_1(M^*)-\lambda_2(M^*) \;>\;0.
\]
Let $M\in\R^{n\times n}$ be symmetric with unit top eigenvector $u$ and suppose that
\[
\langle u,u^*\rangle \ge 0 .
\]
Then
\[
\|u-u^*\|_2 \wle \frac{2^{3/2}\,\|M-M^*\|}{\gamma^*}.
\]
\end{lemma}

\subsection{Row-wise concentration tools}
Fix \(S\subset[n]\) and work under \(\P_S\). Recall that
\[
M = A-\E_0[A],
\qquad
M^* = \E_S[M].
\]
Let
\[
\cE= \binom{[n]}{d},
\qquad
\Ein= \{e\in\cE:e\subset S\},
\qquad
\Eout= \cE\setminus\Ein,
\]
and, for each \(i\in[n]\),
\[
\cE_i= \{e\in\cE:i\in e\},
\qquad
\Eout_i= \Eout\cap \cE_i.
\]

\begin{observation}\label{obs:row_decomp_direction}
Fix \(i\in[n]\) and \(v\in\R^n\). For any \(d\)-set \(e\) with \(i\in e\), define
\begin{equation}
a_e \coloneqq \sum_{j\in e\setminus\{i\}} v_j.
\label{eq:def_ae}
\end{equation}
For each \(e\in\Eout_i\), set
\[
Y_e \coloneqq (H_e-p)a_e,
\qquad
(M-M^*)_{i:}v=\sum_{e\in\Eout_i}Y_e .
\]
Under \(\P_S\), the variables \(\{Y_e\}_{e\in\Eout_i}\) are independent and centered. Moreover, for \(e\in\Eout_i\),
\begin{equation}\label{eq:row_envelope_short}
|Y_e|
=|H_e-p|\,|a_e|
\le |a_e|
\le \sum_{j\in e\setminus\{i\}}|v_j|
\le (d-1)\|v\|_\infty.
\end{equation}
Finally, with
\begin{equation}\label{eq:row_varproxy_short}
\sigma_i^2(v)= \sum_{e\in\Eout_i}\Var_S(Y_e),
\end{equation}
we have, since \(\Var_S(H_e)=p(1-p)\) for \(e\in\Eout\),
\[
\sigma_i^2(v)
=\sum_{e\in\Eout_i} a_e^2\Var_S(H_e)
=p(1-p)\sum_{e\in\Eout_i}a_e^2.
\]
\end{observation}

\begin{lemma}\label{lem:sum-ae2}
Fix \(i\in[n]\) and \(v\in\R^n\). For each \(d\)-set \(e\) with \(i\in e\), let \(a_e\) be as in \eqref{eq:def_ae}. Then
\begin{equation}\label{eq:l2-proxy}
\sum_{e\in\Eout_i} a_e^2 \wle \sum_{e \in \cE_i} a_e^2
\wle
(d-1)\binom{n-2}{d-2} \|v\|_2^2 .
\end{equation}
\end{lemma}

\begin{proof}
Since \(\Eout_i\subset \cE_i\), it suffices to bound \(\sum_{e\in\cE_i}a_e^2\).
For each \(e\in \cE_i\), let \(T_e\coloneqq e\setminus\{i\}\), so \(|T_e|=d-1\) and \(a_e=\sum_{j\in T_e}v_j\).
By Cauchy--Schwarz \eqref{ineq:cs},
\[
\left(\sum_{j\in T_e} v_j\right)^2
\wle
\left(\sum_{j\in T_e} 1^2\right)\left(\sum_{j\in T_e} v_j^2\right)
\weq
(d-1)\sum_{j\in T_e} v_j^2 .
\]
Summing over all \((d-1)\)-subsets \(T_e\subset[n]\setminus\{i\}\) and interchanging the order of summation,
\[
\sum_{e\in\cE_i}a_e^2
\wle
(d-1)\sum_{\substack{T_e\subset[n]\setminus\{i\}\\ |T_e|=d-1}}\ \sum_{j\in T_e}v_j^2
=
(d-1)\binom{n-2}{d-2}\sum_{j\neq i}v_j^2
\wle
(d-1)\binom{n-2}{d-2}\|v\|_2^2,
\]
since each \(j\neq i\) belongs to exactly \(\binom{n-2}{d-2}\) such sets \(T_e\).
\end{proof}

\begin{lemma}[Row concentration in a fixed direction]\label{lem:row-bernstein-general}
Fix \(i\in[n]\), \(v\in\R^n\), and \(c>0\). Let
\[
C\ge \max\left\{ \sqrt{2(c+1)}, \frac{2}{3}(c+1)\right\}.
\]
Then, with \(\P_S\)-probability at least \(1-2n^{-(c+1)}\),
\begin{equation}\label{eq:row-one-i-hp}
|(M-M^*)_{i:} v|
\le
 C\left(V_n\|v\|_2+K_n\|v\|_\infty\right).
\end{equation}
\end{lemma}

\begin{proof}
By Observation~\ref{obs:row_decomp_direction}, \((M-M^*)_{i:}v=\sum_{e\in\Eout_i}Y_e\) with \(\{Y_e\}\) independent and centered.
By \eqref{eq:row_envelope_short},
\[
|Y_e|\le (d-1)\|v\|_\infty=\frac{K_n}{\log n}\,\|v\|_\infty.
\]
Moreover, by \eqref{eq:row_varproxy_short} and Lemma~\ref{lem:sum-ae2},
\[
\sigma_i^2(v)
=p(1-p)\sum_{e\in\Eout_i}a_e^2
\le p(1-p)(d-1)\binom{n-2}{d-2}\,\|v\|_2^2
=\frac{V_n^2}{\log n}\,\|v\|_2^2.
\]
Applying Lemma~\ref{lem:scalar-bernstein} gives, for all \(t>0\),
\[
\P_S\left(|(M-M^*)_{i:} v|\ge t\right)
\le
2\exp\left(
-\frac{t^2}{2\frac{V_n^2}{\log n}\|v\|_2^2+\frac{2}{3}\frac{K_n}{\log n}\|v\|_\infty\, t}
\right).
\]
Set \(\ell\coloneqq (c+1)\log n\) and
\[
t\coloneqq \sqrt{2\left(\frac{V_n^2}{\log n}\|v\|_2^2\right)\ell}
+\frac{2}{3}\left(\frac{K_n}{\log n}\|v\|_\infty\right)\ell
=
\sqrt{2(c+1)}\,V_n\|v\|_2+\frac{2}{3}(c+1)\,K_n\|v\|_\infty.
\]
By Lemma~\ref{lem:bernstein_parameter_ineq}, the Bernstein exponent is at least \(\ell\), hence
\(\P_S(|(M-M^*)_{i:} v|\ge t)\le 2e^{-\ell}=2n^{-(c+1)}\).
Finally, by the choice of \(C\),
\[
t\le C\left(V_n\|v\|_2+K_n\|v\|_\infty\right),
\]
which yields \eqref{eq:row-one-i-hp}.
\end{proof}



\end{document}